\documentclass[11pt]{amsart}
\usepackage{amssymb}
\usepackage{amsmath}
\usepackage{amsfonts}

\newtheorem{theorem}{Theorem}
\theoremstyle{plain}

\newtheorem{definition}{Definition}
\newtheorem{example}{Example}

\newtheorem{lemma}{Lemma}

\newtheorem{proposition}{Proposition}
\newtheorem{remark}{Remark}

\numberwithin{equation}{section}

\begin{document}
\title[Systems of quasilinear elliptic equations]{Existence, positivity and
boundedness of solutions for systems of quasilinear elliptic equations}

\begin{abstract}
This article sets forth results on the existence, positivity and boundedness
of solutions for quasilinear elliptic systems involving $p$-Laplacian and $q$-Laplacian operators. The approach combines Schaefer's fixed point,
comparison principle as well as Moser's iteration procedure.
\end{abstract}

\author{Abdelkrim Moussaoui}
\address{Abdelkrim Moussaoui\\
Biology department\\
A. Mira Bejaia University, Targa Ouzemour 06000 Bejaia, Algeria}
\email{abdelkrim.moussaoui@univ-bejaia.dz}
\author{Jean V\'{e}lin}
\address{Jean V\'{e}lin\\
D\'{e}partement de Math\'{e}matiques et Informatique, Laboratoire LAMIA,
Universit\'{e} des Antilles, Campus de Fouillole 97159 Pointe-\`{a}-Pitre,
Guadeloupe (FWI)}
\email{jean.velin@univ-antilles.fr}
\maketitle

\section{Introduction}

Let $\Omega \subset \mathbb{R}^{N}$ $\left( N\geq 2\right) $ be a bounded
domain with smooth boundary $\partial \Omega $. Given $1<p,q<N$, we consider
the quasilinear elliptic problem%
\begin{equation}
\left\{ 
\begin{array}{ll}
-\Delta _{p}u=f(x,u,v) & \text{in }\Omega \\ 
-\Delta _{q}v=g(x,u,v) & \text{in }\Omega \\ 
u=v=0 & \text{on }\partial \Omega ,%
\end{array}%
\right.  \tag{$P$}  \label{p}
\end{equation}%
where $\Delta _{p}$ and $\Delta _{q}$ stand for the $p$-Laplacian and $q$%
-Laplacian on $W_{0}^{1,p}(\Omega )$ and $W_{0}^{1,q}(\Omega ),$
respectively. The nonlinearities $f,g:\Omega \times \mathbb{R}%
^{2}\rightarrow \mathbb{R}$ in (\ref{p}) are Carath\'{e}odory functions,
that is, $f(\cdot ,s,t),$ $g(\cdot ,s,t)$ are measurable for every $(s,t)\in 
\mathbb{R}^{2}$, and $f(x,\cdot ,\cdot ),$ $g(x,\cdot ,\cdot )$ are
continuous for a.e. $x\in \Omega $.

A solution of (\ref{p}) is understood in the weak sense, which means a pair
of functions $(u,v)\in W_{0}^{1,p}\left( \Omega \right) \times
W_{0}^{1,q}\left( \Omega \right) $ such that 
\begin{equation*}
\left\{ 
\begin{array}{c}
\int_{\Omega }\left\vert \nabla u\right\vert ^{p-2}\nabla u\nabla \varphi \
dx=\int_{\Omega }f(x,u,v)\varphi \ dx \\ 
\int_{\Omega }|\nabla v|^{q-2}\nabla v\nabla \psi \ dx=\int_{\Omega
}g(x,u,v)\psi \ dx%
\end{array}%
\right.
\end{equation*}%
for all $(\varphi ,\psi )\in W_{0}^{1,p}(\Omega )\times W_{0}^{1,q}(\Omega
), $ provided the integrals in the right-hand side of the above equalities
exist.

\bigskip

Quasilinear elliptic systems have been quite intensely investigated in the
literature with various methods. Among them, in \cite{ka,df1,df2,
dtlv1,dtlv2}, the authors take advantage of the variational structure of the
problem to apply variational methods. In \cite{as,sz}, some of these methods
combined to Nehari manifolds are used. Nonvariational problems also have
been widely investigated through topological methods. Namely, we quote
Schaefer's fixed point \cite{jp}, monotonicity method \cite{ct},
Leray-Schauder degree theory \cite{cfmt, cgg, wf}, fixed point index \cite{w}%
, sub-supersolution technics \cite{ah, mps, hs} and blow-up method combined
with a suitable degree argument \cite{cfmt}. We also mention \cite{ghs, hmv,
cm, mm, ghm, t} focusing on the semilinear case of (\ref{p}), that is, when $%
p=q=2$. It is worth noting that the aforementioned works focus on the
following type growth condition%
\begin{equation*}
|sf(x,s,t)|,|tg(x,s,t)|\leq k(x)(|s|^{\gamma }+|t|^{\delta })
\end{equation*}%
where $1<\gamma \leq p,$ $1<\delta \leq q.$

In the present paper, we consider the complementary case in which $%
|sf(x,s,t)|$ and $|tg(x,s,t)|$ satisfy growth condition of type $|s|^{\gamma
}+|t|^{\delta },$ $\gamma \in (p,p^{\star }),$ $\delta \in (q,q^{\star })$,
where $p^{\star }$ and $q^{\star }$ are the Sobolev critical exponents, that
is, $p^{\star }=\frac{Np}{N-p}$ and $q^{\star }=\frac{Nq}{N-q}$. This
represents a serious difficulty to overcome, and is rarely handled in the
literature. Moreover, the difficulty is even more\textbf{\ }stressed
because, on one the hand, no structural assumption is assumed guaranteeing
that the Euler functional associated to problem (\ref{p}) is well defined
and therefore, the variational method cannot be applied. On the other hand,
the sub-supersolution method does not work for problem (\ref{p}) due to of
its noncooperative character. This means that generally the functions $%
f(x,u,\cdot )$ and $g(x,\cdot ,v)$ are not necessarily increasing whenever $%
u,v$ are fixed. It is worth pointing out that no sign condition is required
on the right-hand side nonlinearities and so large classes of quasilinear
problems involving $p$-Laplacian operator can be incorporated in (\ref{p}).

\bigskip

Throughout this paper, we assume that the nonlinear terms $f$ and $g$
satisfy the following assumptions:

\mathstrut

\begin{description}
\item[$($\textrm{H.1}$)$] For $(u,v)\in W_{0}^{1,p}(\Omega )\times
W_{0}^{1,q}(\Omega ),$ 
\begin{equation*}
x\mapsto f(x,u(x),v(x))\in L^{p_{C}^{\prime }}(\Omega ),\,\,x\mapsto
g(x,u(x),v(x))\in L^{q_{C}^{\prime }}(\Omega )
\end{equation*}%
where $p_{C}^{\prime }=\dfrac{Cp}{Cp-1},$ $q_{C}^{\prime }=\dfrac{Cq}{Cq-1}$
and%
\begin{equation}
1<C<\min \{\dfrac{p^{\star }}{p},\,\dfrac{q^{\star }}{q}\}.  \label{10}
\end{equation}
\end{description}

\mathstrut

\begin{description}
\item[$($\textrm{H.2}$)$] There exists a positive real function $k_{p,q}\in
L^{\infty }(\Omega )$ such that%
\begin{equation*}
{\left\vert sf(x,s,t)\right\vert {\wedge }\left\vert tg(x,s,t)\right\vert
\leq }k_{p,q}(x)\left( |s|^{\alpha +1}|t|^{\beta +1}\right) \vee \left(
|s|^{pC}+|t|^{qC}\right) ,
\end{equation*}%
for a.e. $x\in \Omega $ and all $s,t\in 
\mathbb{R}
,$ with%
\begin{equation}
\alpha >-1,\,\,\beta >-1,\text{ \ }\dfrac{\alpha +1}{p}+\dfrac{\beta +1}{q}%
=1.  \label{11}
\end{equation}
\end{description}

Here, for any $w_{1},w_{2}\in 
\mathbb{R}
$, we denote 
\begin{equation*}
|w_{1}|\wedge |w_{2}|:=\max \{|w_{1}|,|w_{2}|\}\text{ \ and \ }|w_{1}|\vee
|w_{2}|:=\min \{|w_{1}|,|w_{2}|\}.
\end{equation*}

\mathstrut

Our main interest in this work consists in getting solutions of system (\ref%
{p}) with additional qualitative properties. Namely, we established the
existence, positivity and boundedness of nontrivial solutions. Our first
main result deals with existence of nontrivial solutions which is stated as
follows.

\begin{theorem}
\label{T1} Under the assumptions $($\textrm{H.1}$)$ and $($\textrm{H.2}$)$
system (\ref{p}) admits at least one nontrivial solution $(u^{\ast },v^{\ast
})$ in $C^{1,\sigma }(\overline{\Omega })\times C^{1,\sigma }(\overline{%
\Omega })$ for certain $\sigma \in (0,1)$.
\end{theorem}

The proof of Theorem \ref{T1} is chiefly based on Schaefer's fixed point
Theorem (see, e.,g. \cite[Theorem 4, Section 9.2.2]{ev}, \cite{sma}), which
guarantees the existence of a weak solution $(u^{\ast },v^{\ast })$ in $%
W_{0}^{1,p}\left( \Omega \right) \times W_{0}^{1,q}\left( \Omega \right) $.
This required Besov spaces involvement, especially the embeddings from Besov
into Sobolev spaces which is one of a significant feature of the present
work. Moreover, we prove there exist two constants $\varepsilon _{1}$ and $%
\Theta $ such that $0<\varepsilon _{1}\leq \Vert u^{\ast }\Vert _{1,p}+\Vert
v^{\ast }\Vert _{1,q}\leq \Theta <+\infty .$ This ensures the nontriviality
character of the obtained solution $(u^{\ast },v^{\ast })$ in $%
W_{0}^{1,p}\left( \Omega \right) \times W_{0}^{1,q}\left( \Omega \right) $.

\mathstrut

The $L^{\infty }$-Boundedness for an arbitrary weak solution of problem (\ref%
{p}) is also provided in the present work. Combined with the regularity
result in \textbf{\ }\cite{tldf}, it ensures in particular that the obtained
solution $(u^{\ast },v^{\ast })$ is bounded in $C^{1,\sigma }(\overline{%
\Omega })\times C^{1,\sigma }(\overline{\Omega })$ for certain $\sigma \in
(0,1)$. Mainly through Moser's iteration process one can prove the next
result.

\begin{theorem}
\label{T3} Under assumptions $($\textrm{H.1}$)$ and $($\textrm{H.2}$)$, all
solutions $(u,v)$ of (\ref{p}) are bounded in $L^{\infty }(\Omega )\times {%
L^{\infty }(\Omega )}$.
\end{theorem}

Another main achievement of our work consists to provide a precise sign
information on solutions of problem (\ref{p}). In this respect, we establish
the existence of a positive solution $(u,v)$ in the sense that both
components $u$ and $v$ are positive. Our argument relies on a comparison
principle based on fibering method due to Pohozaev. However, additional
assumptions on $f$ and $g$ are required and are formulated as follows:

\mathstrut 

\begin{description}
\item[$($\textrm{H.3}$)$] 
\begin{equation*}
\left\{ 
\begin{array}{l}
(f(x,s,t)-f(x,\bar{s},t))(s-\bar{s})\leq 0,\text{ for a.e }x\in \Omega , \\ 
\text{for all }t\in \mathbb{R}\text{, and all }s,{\overline{s}}\in {\mathbb{R%
}}\backslash {\mathbb{\{}}0{\mathbb{\}}},%
\end{array}%
\right. 
\end{equation*}
\begin{equation*}
\left\{ 
\begin{array}{l}
(g(x,s,t)-g(x,s,\bar{t}))(t-\bar{t})\leq 0,\text{ for a.e }x\in \Omega , \\ 
\text{for all }s\in \mathbb{R}\text{, and all }t,{\overline{t}}\in {\mathbb{R%
}}\backslash {\mathbb{\{}}0{\mathbb{\}}}.%
\end{array}%
\right. \text{ }
\end{equation*}
\end{description}

\mathstrut

\begin{description}
\item[$($\textrm{H.4}$)$] There exist functions $a_{p},a_{q}\in L^{\infty
}(\Omega ),b_{p}\in L^{\delta _{p}}(\Omega )$ and $b_{q}\in L^{\delta
_{q}}(\Omega ),$ with $\delta _{p}>N/p$, $\delta _{q}>N/q,$ such that 
\begin{equation*}
f(x,s,t)\geq {a_{p}}(x)s|s|^{{\hat{\alpha}}-1}|t|^{\hat{\beta}+1}+{b_{p}}%
(x)s|s|^{p-2}
\end{equation*}%
and%
\begin{equation*}
g(x,s,t)\geq {a_{q}}(x)|s|^{{\hat{\alpha}}+1}|t|^{\hat{\beta}-1}t+{b_{q}}%
(x)t|t|^{q-2},
\end{equation*}%
for a.e. $x\in \Omega ,\,$and all $(s,t)\in \mathbb{%
\mathbb{R}
}\times \mathbb{%
\mathbb{R}
}$, with 
\begin{equation}
\begin{array}{l}
\hat{\alpha}+1\neq p,\text{ }\hat{\beta}+1\neq q\text{ \ and \ }\dfrac{\hat{%
\alpha}+1}{p^{\ast }}+\dfrac{\hat{\beta}+1}{q^{\ast }}<1.%
\end{array}
\label{12}
\end{equation}
\end{description}

\mathstrut

The obtained result on positivity property is formulated as follows.

\begin{theorem}
\label{T2}Assume that $($\textrm{H.1}$)$ - $($\textrm{H.4}$)$ hold. Then
problem (\ref{p}) possesses a positive solution $(u,v)$ in $C^{1,\sigma }(%
\overline{\Omega })\times C^{1,\sigma }(\overline{\Omega })$ for certain $%
\sigma \in (0,1)$.\bigskip
\end{theorem}

We indicate an example showing the applicability of Theorems \ref{T1}, \ref%
{T3} and \ref{T2}.

\begin{example}
Consider the functions $f,g:\Omega \times 
\mathbb{R}
^{2}\rightarrow 
\mathbb{R}
$ defined by%
\begin{equation*}
\begin{array}{l}
f(x,s,t)=\dfrac{k_{p,q}(x)}{2}\left( \dfrac{1}{1+|s|^{-\alpha }}%
+h_{1}(s)\right) |t|^{\beta +1}%
\end{array}%
\end{equation*}%
and%
\begin{equation*}
\begin{array}{l}
g(x,s,t)=\dfrac{k_{p,q}(x)}{2}|s|^{\alpha +1}\left( \dfrac{1}{1+|t|^{-\beta }%
}+h_{2}(t)\right) ,%
\end{array}%
\end{equation*}%
where $k_{p,q}(\cdot )$ is a bounded positive function in $L^{\infty
}(\Omega )$ and%
\begin{equation*}
h_{1}(s)=\left\{ 
\begin{array}{cc}
s^{\alpha } & \text{ if }s\geq 1 \\ 
1 & \text{ if }s\leq 1%
\end{array}%
\right. ,\text{ \ }h_{2}(t)=\left\{ 
\begin{array}{cc}
t^{\beta } & \text{ if }t\geq 1 \\ 
1 & \text{ if }t\leq 1,%
\end{array}%
\right. 
\end{equation*}%
with%
\begin{equation*}
-1<\alpha ,\beta <0,\text{ }\dfrac{\alpha +1}{p}+\dfrac{\beta +1}{q}=1.
\end{equation*}%
It is straightforward to check that conditions $($\textrm{H.1}$)$-$($\textrm{%
H.4}$)$ are verified. Consequently, Theorems \ref{T1}, \ref{T3} and \ref{T2}
are applicable providing positive and bounded solutions for system (\ref{p})
with equations whose right-hand sides are given through the preceding
functions $f$ and $g$.
\end{example}

The rest of the paper is organized as follows. Section \ref{S1} contains the
existence of nontrivial solutions for problem (\ref{p}). Section \ref{S2}
deals with the positivity property while section \ref{S3} focuses on $%
L^{\infty }$-boundedness of solutions.

\section{Existence of solutions}

\label{S1}

Given a number $1<p<\infty $, the space $L^{p}(\Omega )$ is endowed with the
norm $\left\Vert u\right\Vert _{p}=(\int_{\Omega }|u|^{p}\ dx)^{1/p}$, while
on $W_{0}^{1,p}\left( \Omega \right) $ we consider the norm $\left\Vert
u\right\Vert _{1,p}=\left( \int_{\Omega }\left\vert \nabla u\right\vert
^{p}\ dx\right) ^{1/p}$. Throughout this paper, $p^{\prime }=\dfrac{p}{p-1}$
and $p^{\star }=\frac{Np}{N-p}$ are the conjugate and the Sobolev critical
exponents, respectively, while $\langle ,\,\rangle _{-1,1}$ denotes the
duality brackets between the space $W_{0}^{1,p}(\Omega )$ and its
topological dual $W^{1,p^{\prime }}(\Omega )$.

\begin{remark}
\label{R1} Fix $(u,v)$ in $W_{0}^{1,p}(\Omega )\times W_{0}^{1,q}(\Omega ).$
By (\ref{10}), (\ref{11}) together with Young's and Jensen's inequalities it
holds 
\begin{equation*}
{\int_{\Omega }}|u|^{\alpha +1}|v|^{\beta +1}dx\leq \Vert u\Vert
_{p}^{p}+\Vert v\Vert _{q}^{q}\leq \left( \Vert u\Vert _{p}^{pC}+1\right)
+\left( \Vert v\Vert _{q}^{qC}+1\right) .
\end{equation*}%
Then, Poincar\'{e}'s inequality implies 
\begin{equation*}
({\int_{\Omega }}uf(x,u,v)dx)\wedge (\int_{\Omega }vg(x,u,v)dx)\leq
const.\left( \Vert u\Vert _{1,p}^{pC}+\Vert v\Vert _{1,q}^{qC}+1\right) .
\end{equation*}%
Hence, by assumptions (\ref{10}) - (\ref{11}), Sobolev embedding Theorems
are applicable.
\end{remark}

We will also make use of Besov space $B_{\mathbf{\mathrm{p}}}^{\sigma
,p}(\Omega ),$ for $1\leq \mathbf{\mathrm{p}}\leq \infty ,$ defined as
follows%
\begin{equation*}
B_{\mathbf{\mathrm{p}}}^{\sigma ,p}(\Omega )=\left[ W^{E(\sigma
)+1,p}(\Omega ),W^{E(\sigma ),p}(\Omega )\right] _{E(\sigma )+1-\sigma ,%
\mathbf{\mathrm{p}}},
\end{equation*}%
where $E(\sigma )$ designates the entire part of the real $\sigma $ (see 
\cite{s2}). Note that for a bounded domain $\Omega $ the above definition
remains valid for $W_{0}^{s,p}(\Omega )$ instead of $W^{s,p}(\Omega )$.

\begin{lemma}
\label{Binf} The embeddings $B_{\infty }^{1+C,p}(\Omega )\hookrightarrow
W_{0}^{1,p}(\Omega )$ and $B_{\infty }^{1+C,q}(\Omega )\hookrightarrow
W_{0}^{1,q}(\Omega )$ are compact.
\end{lemma}

\begin{proof}
Observe that \cite[Proposition 4.3]{gm} is applicable due to the compactness
of the embedding $W_{0}^{E(\sigma )+1,p}(\Omega )\hookrightarrow
W_{0}^{E(\sigma ),p}(\Omega )$ (see \cite[Theorem 6.2]{ad} with $\Omega
_{0}=\Omega ,$ $k=N,$ $j=E(\sigma ),$ $\mathbf{\mathrm{p}}=p$ and $m=1$).
Thus the embedding $B_{\mathbf{\mathrm{p}}}^{\sigma ,p}(\Omega
)\hookrightarrow W_{0}^{E(\sigma ),p}(\Omega )$ is compact and therefore,
the embedding $B_{\infty }^{1,p}(\Omega )$ (resp. $B_{\infty }^{1,q}(\Omega
) $) in $W_{0}^{1,p}(\Omega )$ (resp. $W_{0}^{1,q}(\Omega )$) is compact. By
the iteration process, we deduce that the embedding $B_{\infty
}^{1+C,p}(\Omega )$ (resp. $B_{\infty }^{1+C,q}(\Omega )$) in $%
W_{0}^{1,p}(\Omega )$ (resp. $W_{0}^{1,q}(\Omega )$) is also compact.
\end{proof}

In the sequel, We denote $t^{\pm }:=\max \{0,\pm t\}$ and we set $%
X:=W_{0}^{1,p}(\Omega )\times W_{0}^{1,q}(\Omega )$ equipped with the norm $%
\left\Vert (u,v)\right\Vert _{X}=\left\Vert u\right\Vert _{1,p}+\left\Vert
v\right\Vert _{1,q}$.

\bigskip

In this section we focus on the existence of solutions for system (\ref{p}%
).\ Our approach is based on the following Schaefer's fixed point theorem
(see e.g., \cite[p.29]{sma} and \cite[chap. 9.2.2]{ev}).

\begin{theorem}
\label{TT}Assume that $T:X\longrightarrow {X}$ is a continuous mapping which
is compact on each bounded subset $\mathcal{B}$ of $X.$ Then, either the
equation $x=\tau {Tx}$ has a solution for $\tau =1$ or the set of all
solution $x$ is unbounded for $0<\tau <1.$\bigskip
\end{theorem}

Let $\mathcal{T}:X\rightarrow X$ be the nonlinear operator such that $%
\mathcal{T}(u,v)=(z,w)$, where $(z,w)$ is required to satisfy 
\begin{equation*}
(P_{z,w})\mathcal{\qquad }\left\{ 
\begin{array}{ll}
-\Delta _{p}z=f(x,u,v) & \hbox{ in }\Omega \\ 
-\Delta _{q}w=g(x,u,v) & \hbox{ in }\Omega \\ 
z=w=0 & \hbox{ on }\partial \Omega .%
\end{array}%
\right.
\end{equation*}%
By $($\textrm{H.1}$)$, the unique solvability of $(z,w)$ in $(\mathcal{P}%
_{z,w})$ is readily derived from Minty-Browder Theorem (see, e.g, \cite{B}).
Thus, the operator $\mathcal{T}$ is well defined.

\begin{lemma}
\label{L3} Under assumptions $($\textrm{H.1}$)$ and $($\textrm{H.2}$)$ the
operator $\mathcal{T}$ is continuous.
\end{lemma}

\begin{proof}
Let $(u_{n},v_{n})\in X,$ $(z_{n},w_{n})=\mathcal{T}(u_{n},v_{n})$ with 
\begin{equation}
(u_{n},v_{n})\rightarrow (u,v)\text{ in }X.  \label{1}
\end{equation}%
Set $(z,w)=\mathcal{T}(u,v)$ and 
\begin{equation}
\mathbf{f}_{n}=f(\cdot ,u_{n},v_{n}),\text{ \ }\mathbf{g}_{n}=g(\cdot
,u_{n},v_{n})\text{\ \ }\mathbf{f}=f(\cdot ,u,v),\text{ \ }\mathbf{g}%
=g(\cdot ,u,v).  \label{2}
\end{equation}%
The continuity of $\mathcal{T}$ follows if we show that 
\begin{equation*}
\mathbf{f}_{n}\rightarrow \mathbf{f}\text{ \ in }L^{p_{C}^{\prime }}(\Omega )%
\text{ \ and \ }\mathbf{g}_{n}\rightarrow \mathbf{g}\text{ \ in }%
L^{q_{C}^{\prime }}(\Omega ).
\end{equation*}
Let $\{\mathbf{f}_{n_{k}}\}$ be a subsequence. By (\ref{1}), it follows that 
$u_{n_{k}}\rightarrow u$ in $L^{pC}(\Omega )$ and $v_{n_{k}}\rightarrow v$
in $L^{qC}(\Omega )$. The continuity of $f$ implies $\mathbf{f}%
_{n_{k}}(x)\rightarrow \mathbf{f}(x)$ for a.e $x$ in $\Omega .$

On the other hand, one can extract subsequences $u_{n_{k_{l}}}$ and $%
v_{n_{k_{l}}}$ such that $u_{n_{k_{l}}}(x)\rightarrow u(x)$ and $%
v_{n_{k_{l}}}(x)\rightarrow v(x)$ for \textit{a.e} $x$ in $\Omega .$
Moreover, there exist positive functions $U\in L^{pC}(\Omega )$ and $V\in
L^{qC}(\Omega )$ such that $|u_{n_{k_{l}}}(x)|\leq U(x)$ and $%
|v_{n_{k_{l}}}(x)|\leq V(x),$ \textit{a.e.} $x$ in $\Omega $ and all $l\in 
\mathbb{N}
$. Then, the continuity of $f$ gives 
\begin{equation*}
|\mathbf{f}_{n_{k_{l}}}(x)|\leq {\sup_{\substack{ -U(x)\leq s\leq U(x) \\ %
-V(x)\leq t\leq V(x)}}}\left\vert {f}(x,s,t)\right\vert ,\,\,\,\,\forall
l\in 
\mathbb{N}
,\,\,\,\hbox{ a.e }x\in \Omega .
\end{equation*}%
Owing to Lebesgue's dominated convergence Theorem, we conclude that 
\begin{equation*}
{\lim_{l\rightarrow +\infty }}\int_{\Omega }|\mathbf{f}_{n_{k_{l}}}(x)-%
\mathbf{f}(x)|^{p_{C}^{\prime }}dx=0.
\end{equation*}%
From the \textit{Urysohn's subsequence principle} (see, e.g., \cite[%
Proposition A.6, p.179]{so} or \cite{df3}), it follows that all the sequence
($\mathbf{f}_{n}$) obeys to 
\begin{equation}
{\lim_{n\rightarrow +\infty }}\left\Vert \mathbf{f}_{n}-\mathbf{f}%
\right\Vert _{p_{C}^{\prime }}^{p_{C}^{\prime }}dx=0.  \label{bigf}
\end{equation}%
Multiplying each equation in $(\mathcal{P}_{z,w})$ by $z_{n}-z$ and $w_{n}-w$%
, respectively, and integrating over $\Omega $, one gets 
\begin{equation*}
\Vert z_{n}-z\Vert _{1,p}^{p}\leq {\int_{\Omega }}\left\vert \mathbf{f}%
_{n}(x)-\mathbf{f}(x)\right\vert |z_{n}-z|dx.
\end{equation*}%
By H\"{o}lder's inequality together with the embedding $W_{0}^{1,p}(\Omega
)\hookrightarrow L^{pC}(\Omega )$, one can find a constant $c_{p}>0$ such
that 
\begin{equation*}
\Vert z_{n}-z\Vert _{1,p}^{p-1}\leq c_{p}\left( {\int_{\Omega }}\left\vert 
\mathbf{f}_{n}(x)-\mathbf{f}(x)\right\vert ^{p_{C}^{\prime }}dx\right)
^{1/p_{C}^{\prime }}.
\end{equation*}%
Thanks to Lemma \ref{bigf}, we conclude that $z_{n}\rightarrow z$ in $%
W_{0}^{1,p}(\Omega ).$ A quite similar argument provides $w_{n}\rightarrow w$
in $W_{0}^{1,q}(\Omega )$ and therefore, $\mathcal{T}(u_{n},v_{n})%
\rightarrow \mathcal{T}(u,v)$ in $X.$ This ends the proof.
\end{proof}

\bigskip

\begin{lemma}
\label{L2} Assume $($\textrm{H.1}$)$ and $($\textrm{H.2}$)$ hold. Then $%
\mathcal{T}$ is compact.
\end{lemma}

\begin{proof}
For a bounded sequence $(u_{n},v_{n})_{n}$ in $X$ and $\mathbf{f}_{n},%
\mathbf{g}_{n}$ defined in (\ref{2}), let us show that there exists a
subsequence $(u_{n_{k}},v_{n_{k}})_{n_{k}}$ such that 
\begin{equation*}
|\mathbf{f}_{n_{k}}(x)|,|\mathbf{g}_{n_{k}}(x)|\leq const.,\text{ for a.e}%
\mathit{.\ }x\in \Omega ,\forall n\in \mathbb{N},.
\end{equation*}
From (\ref{10}), the embeddings $W_{0}^{1,p}(\Omega )\hookrightarrow
L^{pC}(\Omega )$ and $W_{0}^{1,q}(\Omega )\hookrightarrow L^{qC}(\Omega )$
are compact. By Rellich-Kondrachov compactness Theorem, along a relabeled
subsequence, $(u_{n},v_{n})_{n}$ converges strongly in $L^{pC}(\Omega
)\times L^{qC}(\Omega )$. Consequently, we can extract subsequence $%
(u_{n_{k}},v_{n_{k}})\rightarrow (\tilde{u},\tilde{v})$ a.e. in $\Omega .$
Exploiting the continuity of $f$ and $g$, we derive that $\mathbf{f}_{n_{k}}$
and $\mathbf{g}_{n_{k}}$ converge to $\tilde{\mathbf{f}}(x)=f(x,\tilde{u},%
\tilde{v})$ and $\tilde{\mathbf{g}}(x)=g(x,\tilde{u},\tilde{v})$ \textit{a.e.%
} in $\Omega $, as well as, $\mathbf{f}_{n_{k}}$, $\mathbf{g}_{n_{k}}$ are
bounded in $L^{p_{C}^{\prime }}(\Omega )$ and $L^{q_{C}^{\prime }}(\Omega ),$
respectively.

Set $(z_{n_{k}},w_{n_{k}})=\mathcal{T}(u_{n_{k}},v_{n_{k}})$. We claim that $%
z_{n_{k}}$ is bounded in the Besov space $B_{\infty }^{1+\frac{1}{p-1}%
,p}(\Omega )$ (resp. $B_{\infty }^{1+(p-1),p}(\Omega )$) if $p\geq 2$ (resp. 
$p<2$). Indeed, we will apply \cite{s2} (precisely, (14), (15) in Lemma 1,
and (22), (25) in the proof of Theorem 1) to $z_{n_{k}},$ $\mathbf{f}_{n_{k}}
$ and $\mathbf{f}_{\infty }={\lim }\mathbf{f}_{n_{k}}$ a.e in $\Omega $.

Let $h$ in $[0,1]$ and $\theta $ in $\mathcal{D}(\Omega )^{N}$. By using 
\cite[(2.8) in Lemma 1.1]{s1}, there exists a positive constant $c,$
independent of $n_{k}$ and $h,$ such that%
\begin{equation*}
\left\Vert (\mathbf{f}_{n_{k}}-\mathbf{f}_{\infty })\circ {e^{h\theta }}%
\right\Vert _{p_{C}^{\prime }}\leq c\text{ }{\sup_{\Omega }}\left\vert
Jac(e^{-h\theta })\right\vert ^{1/p_{C}^{\prime }}\left\Vert \mathbf{f}%
_{n_{k}}-\mathbf{f}_{\infty }\right\Vert _{p_{C}^{\prime }},
\end{equation*}%
where $Jac(e^{-h\theta })$ denotes the jacobian of the map $\theta \mapsto
e^{-\theta }.$ It follows that 
\begin{equation}
\left\Vert (\mathbf{f}_{n_{k}}-\mathbf{f}_{\infty })\circ {e^{h\theta }}-(%
\mathbf{f}_{n_{k}}-\mathbf{f}_{\infty })\right\Vert _{p_{C}^{\prime }}\leq
\left( c\text{ }{\sup_{\Omega }}\left\vert Jac(e^{-h\theta })\right\vert
^{1/p_{C}^{\prime }}+1\right) \left\Vert \mathbf{f}_{n_{k}}-\mathbf{f}%
_{\infty }\right\Vert _{p_{C}^{\prime }}.  \label{jac}
\end{equation}%
Consequently, for $k$ sufficiently large, $\left\Vert \mathbf{f}_{n_{k}}-%
\mathbf{f}_{\infty }\right\Vert _{p_{C}^{\prime }}$ tends to $0.$ So, there
exists a constant $c_{2}>0,$ independent of $h$ and $n,$ such that for $%
0\leq h\leq 1,$ one has%
\begin{equation*}
\Vert z_{n_{k}}\circ {e^{h\theta }}-z_{n_{k}}\Vert _{1,p}\leq \left\{ 
\begin{array}{ll}
c_{2}h^{1/(p-1)} & \hbox{ if }p\geq 2, \\ 
c_{2}h^{p-1}(1+h^{2-p}) & \hbox{ if }p<2.%
\end{array}%
\right.
\end{equation*}%
Therefore 
\begin{equation*}
{\sup_{0\leq h\leq 1}}\dfrac{\Vert z_{n_{k}}\circ {e^{h\theta }}%
-z_{n_{k}}\Vert _{1,p}}{h^{1/(p-1)}}\leq c_{2},\,\,\,\,\hbox{ for }p\geq 2
\end{equation*}%
and 
\begin{equation*}
{\sup_{0\leq h\leq 1}}\dfrac{\Vert z_{n_{k}}\circ {e^{h\theta }}%
-z_{n_{k}}\Vert _{1,p}}{h^{p-1}}\leq 2c_{2},\,\,\,\,\hbox{ for }p<2,
\end{equation*}%
which clearly means that $z_{n_{k}}$ is bounded in the Besov space $%
B_{\infty }^{1+\frac{1}{p-1},p}(\Omega )$ (resp. $B_{\infty
}^{1+(p-1),p}(\Omega )$) if $p\geq 2$ (resp. $p<2$). This proves the claim.

Arguing similarly we infer that $w_{n_{k}}$ is bounded in the Besov space $%
B_{\infty }^{1+\frac{1}{q-1},q}(\Omega )$ (resp. $B_{\infty
}^{1+(q-1),q}(\Omega )$) if $q\geq 2$ (resp. $q<2$).

Finally, thanks to Lemma \ref{Binf}, one can extract a subsequence still
denoted by $(z_{n},w_{n})$ which converges strongly in $X.$ Thus, the
operator $\mathcal{T}$ is compact, ending the proof of Lemma.
\end{proof}

\bigskip

Next, to implement Schaefer's Theorem, let us introduce, for $\tau \in
(0,1], $ the auxiliary problem%
\begin{equation*}
(\mathcal{P}_{\tau })\mathcal{\qquad }\left\{ 
\begin{array}{ll}
-\Delta _{p}{u}=\tau {f}(x,{u},{v}) & \hbox{ in }\Omega \\ 
-\Delta _{q}{v}=\tau {g(x,{u},{v})} & \hbox{ in }\Omega \\ 
{u}={v}=0 & \hbox{ on }\partial \Omega .%
\end{array}%
\right.
\end{equation*}%
According to the definition of the operator $\mathcal{T}$, system $(\mathcal{%
P}_{\tau })$ may be formulated as $\tau \mathcal{T}{(u,v)}=(u,v).$

\begin{proposition}
\label{Tta} Assume $($\textrm{H.1}$)$ and $($\textrm{H.2}$)$ hold. Given $%
\tau \in (0,1],$ let $(u_{\tau },v_{\tau })$ be such that $(u_{\tau
},v_{\tau })=\tau \mathcal{T}(u_{\tau },v_{\tau }).$ Then there is a
constant $\Theta >0,$ independent of $\tau $, such that $\Vert (u_{\tau
},v_{\tau })\Vert _{X}\leq \Theta $. Moreover, one can find a constant $%
\varepsilon _{0}>0$ such that system $(\mathcal{P}_{\tau })$ has no
solutions on $\partial \mathcal{O}$, where 
\begin{equation*}
\mathcal{O}=\left\{ (u,v)\in X:\,{\varepsilon _{0}}/2<\Vert (u,v)\Vert
_{X}<2\Theta \right\} .
\end{equation*}
\end{proposition}

\begin{proof}
Arguing by contradiction, let $(u_{\tau },v_{\tau })$ be an unbounded
solution of $(\mathcal{P}_{\tau })$ in $X.$ Multiplying the first and the
second equation in $(\mathcal{P}_{\tau })$ by $\dfrac{u_{\tau }}{\Vert
u_{\tau }\Vert _{1,p}^{p}}$ and $\dfrac{v_{\tau }}{\Vert v_{\tau }\Vert
_{1,q}^{q}}$, respectively, one has%
\begin{equation*}
\Vert u_{\tau }\Vert _{1,p}^{p}=\tau {\int_{\Omega }}u_{\tau }f(x,u_{\tau
}(x),v_{\tau }(x))dx\text{ \ and \ }\Vert v_{\tau }\Vert _{1,q}^{q}=\tau {%
\int_{\Omega }}v_{\tau }g(x,u_{\tau }(x),v_{\tau }(x))dx.
\end{equation*}%
Employing $($\textrm{H.2}$)$ it follows that 
\begin{equation}
\begin{array}{l}
1=\tau \dfrac{{\int_{\Omega }}u_{\tau }f(x,u_{\tau },v_{\tau })+v_{\tau
}g(x,u_{\tau },v_{\tau })dx}{\Vert u_{\tau }\Vert _{1,p}^{p}+\Vert v_{\tau
}\Vert _{1,q}^{q}} \\ 
\leq \tau {k_{p,q}}\dfrac{{\int_{\Omega }}u_{\tau }^{\alpha +1}v_{\tau
}^{\beta +1}dx}{\dfrac{\alpha +1}{p}\Vert u_{\tau }\Vert _{1,p}^{p}+\dfrac{%
\beta +1}{q}\Vert v_{\tau }\Vert _{1,q}^{q}}\leq \tau \dfrac{k_{p,q}}{%
\lambda _{p,q}},%
\end{array}
\label{3}
\end{equation}%
where $\lambda _{p,q}$ is the first eigenvalue for a nonlinear elliptic
system with Dirichlet boundary condition that can be characterized by%
\begin{equation*}
\begin{array}{c}
\lambda _{p,q}=\inf_{(u_{\tau },v_{\tau })\in X\backslash \{0\}}\frac{\frac{1%
}{p}\Vert u_{\tau }\Vert _{1,p}^{p}+\frac{1}{q}\Vert v_{\tau }\Vert
_{1,q}^{q}}{{\int_{\Omega }}u_{\tau }^{\alpha +1}v_{\tau }^{\beta +1}dx}%
\end{array}%
\end{equation*}%
(see \cite{dtln}). Then, taking $\Vert u_{\tau }\Vert _{1,p}^{p}+\Vert
v_{\tau }\Vert _{1,q}^{q}$ large enough leads to $\tau \dfrac{k_{p,q}}{%
\lambda _{p,q}}\rightarrow 0,$ which contradicts (\ref{3}). Consequently,
there exists a constant $\Theta >0$ such that all solutions $(u_{\tau
},v_{\tau })$ of the equation $(u,v)=\tau {T}(u,v),$ with $\tau \in (0,1],$
verify 
\begin{equation}
\Vert (u_{\tau },v_{\tau })\Vert _{X}\leq \Theta .  \label{4}
\end{equation}%
Now, we show the second part of the Proposition \ref{Tta}. Set $\Omega
_{2}=\Omega \backslash \Omega _{1}$ with%
\begin{equation*}
\Omega _{1}=\left\{ x\in \Omega :|u_{\tau }(x)|^{pC}+|v_{\tau
}|^{qC}(x)<|u_{\tau }(x)|^{\alpha +1}|v_{\tau }(x)|^{\beta +1}\right\} .
\end{equation*}%
Then 
\begin{equation*}
\begin{array}{l}
\Vert u_{\tau }\Vert _{1,p}^{p}=\tau {\int_{\Omega }}u_{\tau }f(u_{\tau
},v_{\tau })dx \\ 
\leq K\left[ {\int_{\Omega _{1}}}\left( |u_{\tau }|^{pC}+|v_{\tau
}|^{qC}\right) dx+{\int_{\Omega _{2}}}|u_{\tau }|^{\alpha +1}|v_{\tau
}|^{\beta +1}dx\right]  \\ 
\leq K{\int_{\Omega }}\left( |u_{\tau }|^{pC}+|v_{\tau }|^{qC}\right) dx\leq
K(\Vert u_{\tau }\Vert _{pC}^{pC}+\Vert v_{\tau }\Vert _{qC}^{qC}) \\ 
\leq K(C_{p}^{pC}\Vert u_{\tau }\Vert _{1,p}^{pC}+C_{q}^{qC}\Vert v_{\tau
}\Vert _{1,q}^{qC}),%
\end{array}%
\end{equation*}%
where $C_{p}$ and $C_{q}$ are the best constant in the continuous embedding $%
W_{0}^{1,p}(\Omega )\hookrightarrow L^{pC}(\Omega )$ and $W_{0}^{1,q}(\Omega
)\hookrightarrow L^{qC}(\Omega )$. Arguing similarly with the component $%
v_{\tau },$ one gets 
\begin{equation*}
\Vert v_{\tau }\Vert _{1,q}^{q}\leq K(C_{p}^{pC}\Vert u_{\tau }\Vert
_{1,p}^{pC}+C_{q}^{qC}\Vert v_{\tau }\Vert _{1,q}^{qC}).
\end{equation*}%
Then, for any $\tau \in \lbrack 0,1[$, it follows that%
\begin{equation*}
0\leq \Vert u_{\tau }\Vert _{1,p}^{p}(\Vert u_{\tau }\Vert _{1,p}^{p(C-1)}-%
\dfrac{1}{2KC_{p}^{pC}})+\Vert v_{\tau }\Vert _{1,q}^{p}(\Vert v_{\tau
}\Vert _{1,q}^{q(C-1)}-\dfrac{1}{2KC_{q}^{qC}}).
\end{equation*}%
Setting $\varepsilon _{0}={(2K)^{-1}}{\min \{C_{p}^{p(1-C)},C_{q}^{q(1-C)}\}}
$ one derives 
\begin{equation}
\Vert (u_{\tau },v_{\tau })\Vert _{X}\geq \varepsilon _{0}.  \label{5}
\end{equation}%
Consequently, according to (\ref{4}) and (\ref{5}), it is readily seen that
the solutions set of the equation $(u_{\tau },v_{\tau })=\tau {T}(u_{\tau
},v_{\tau })$ verifying $\Vert (u_{\tau },v_{\tau })\Vert _{X}=\varepsilon
_{0}/2$ or $\Vert (u_{\tau },v_{\tau })\Vert _{X}=2\Theta $ is empty.
Namely, system $(\mathcal{P}_{\tau })$ doesn't admit a solution on the
boundary $\partial \mathcal{O}$ for all $\tau \in (0,1]$. This ends the
proof.
\end{proof}

Now we are ready to prove our existence result.

\begin{proof}[\textbf{Proof of Theorem \protect\ref{T1}}]
The proof is a consequence of Lemmas \ref{L3} and \ref{L2} together with
Proposition \ref{Tta}. Hence, owing to Theorem \ref{TT} one concludes that
system (\ref{p}) admits at least a solution $(u^{\star },v^{\star })$ in $X$
satisfying $\varepsilon _{0}\leq \Vert (u^{\star },v^{\star })\Vert _{X}\leq
\Theta $ for certain\ positive constants $\varepsilon _{0}$ and $\Theta $.
Moreover, regularity results due to Tolksdorf \cite{tldf} together with
Theorem \ref{T3} ensure that $(u^{\star },v^{\star })\in C^{1,\sigma }(%
\overline{\Omega })\times C^{1,\sigma }(\overline{\Omega })$ for certain $%
\sigma \in (0,1).$
\end{proof}

\section{Positivity}

\label{S2}

In this section, we show the positivity of the obtained solution $(u^{\star
},v^{\star })$ stated in Theorem \ref{T2}. Our approach is chiefly based on
comparison arguments. To do so, let us recall the following results due to
Pohozaev in \cite[Theorems 5.4.2, 5.5.2 and 5.6.1]{p} (see also \cite[
Theorems 3.4.2, 3.5.2 and 3.6.1]{ch}) for the Dirichlet problem 
\begin{equation*}
(\mathcal{P}_{\lambda })\mathcal{\qquad }\left\{ 
\begin{array}{ll}
-\Delta _{p}{u}=\lambda b(x)|u|^{p-2}u+a(x)|u|^{q-2}u & \hbox{ in }\Omega 
\\ 
{u}=0 & \hbox{ on }\partial \Omega ,%
\end{array}%
\right. 
\end{equation*}%
where $a,b\in L^{\infty }(\Omega )$ and $1<p<q<p^{\star }$.

\begin{proposition}
\label{P1}Let $\lambda _{1}$ be the first eigenvalue of $-\Delta _{p}$ and
let $\phi _{1}$ the corresponding eigenfunction.

\begin{description}
\item[$(\mathrm{1})$] Assume $0\leq \lambda <\lambda _{1}$. Then the problem 
$(\mathcal{P}_{\lambda })$ has at least one positive weak solution $u\in
W_{0}^{1,p}(\Omega )\cap L^{\infty }(\Omega ).$ Moreover, there exists $%
\sigma \in (0,1)$ such that $u\in C_{loc}^{1,\sigma }(\Omega ).$

\item[$(\mathrm{2})$] Assume $0\leq \lambda =\lambda _{1}$ and ${%
\int_{\Omega }}a(x)\phi _{1}^{q}dx<0.$ Then, the problem $(\mathcal{P}%
_{\lambda })$ has at least one positive weak solution $u\in
W_{0}^{1,p}(\Omega )\cap L^{\infty }(\Omega ).$ Moreover, $u\in
C_{loc}^{1,\sigma }(\Omega )$ for certain $\sigma \in (0,1).$

\item[$(\mathrm{3})$] Assume $0\leq \lambda _{1}<\lambda $ and ${%
\int_{\Omega }}a(x)\phi _{1}^{q}dx<0$. Then, there exists $\varepsilon >0$
such that for $\lambda _{1}<\lambda <\lambda _{1}+\varepsilon ,$ problem $(%
\mathcal{P}_{\lambda })$ admits two positive weak solutions in $%
W_{0}^{1,p}(\Omega )\cap L^{\infty }(\Omega )$ and each of them belongs to $%
C_{loc}^{1,\sigma }(\Omega ),$ with $\sigma \in (0,1).$
\end{description}
\end{proposition}

\mathstrut 

We also recall the following definition of conditional critical point.

\begin{definition}
\label{ccpoint}(\cite{ch}, \cite{krv}) Let $\Sigma $ be a real Banach space
and let $h:\Sigma \rightarrow \mathbb{R}$ be a functional such that $h$ is
of class $C^{1}(\Sigma \backslash \{0\})$ and $\tilde{h}:\mathbb{R}%
\backslash \{0\}\times \Sigma \rightarrow \mathbb{R},$ $\tilde{h}(\lambda
,v)=h(\lambda {v}).$ Set $S=\{u\in \Sigma :\Vert u\Vert =1\}.$

A point $(\lambda ,v)\in \mathbb{R}\setminus \{0\}\times {S}$ is a
conditional critical point of the function $\tilde{h}$ if 
\begin{equation*}
-\tilde{h}(\lambda ,v)\in N_{\mathbb{R}\backslash \{0\}\times {S}}(\lambda
,v),
\end{equation*}%
where $N_{\mathbb{R}\backslash \{0\}\times {S}}(\lambda ,v)$ is the normal
cone to the set $\mathbb{R}\backslash \{0\}\times {S}$ at the point $%
(\lambda ,v).$
\end{definition}

\bigskip

In what follows, we denote by $\lambda _{b_{p}}>0$ the first $p$-Laplacian
eigenvalue associated to the weight $b_{p}$%
\begin{equation*}
\lambda _{b_{p}}={\inf_{\overset{z\in W_{0}^{1,p}(\Omega ),}{\int_{\Omega }{%
b_{p}}(x)|z|^{p}dx>0}}}\dfrac{{\int_{\Omega }}|\nabla z|^{p}dx}{{%
\int_{\Omega }}{b_{p}}(x)|z|^{p}dx},
\end{equation*}%
where functions ${a_{p}}(\cdot )$ and $b_{p}(\cdot )$ are defined in $($%
\textrm{H.4}$)$. Let $A_{p}$ and $B_{p}$ be the following applications:%
\begin{equation}
A_{p}(u)={\int_{\Omega }}{a_{p}}(x)|u|^{\hat{\alpha}+1}|v^{\star }|^{\hat{%
\beta}+1}dx,\,\,\,\ \ \ B_{p}(u)={\int_{\Omega }}b_{p}(x)|u|^{p}dx.
\label{ABu}
\end{equation}

\begin{lemma}
\label{AB} $A_{p}$ and $B_{p}$ are weakly continuous in $W_{0}^{1,p}(\Omega
).$
\end{lemma}

\begin{proof}
Let $u$ and $u_{n}$ in $W_{0}^{1,p}(\Omega )$ such that%
\begin{equation}
u_{n}\rightharpoonup u\text{ \ in }W_{0}^{1,p}(\Omega ).  \label{7}
\end{equation}%
We claim that $A_{p}(u_{n})$ tends to $A_{p}(u).$ Indeed, writing 
\begin{equation*}
A_{p}(u_{n})={\int_{\Omega }}a_{p}^{+}(x)|v^{\star }|^{\hat{\beta}%
+1}|u_{n}|^{\hat{\alpha}+1}dx-{\int_{\Omega }}a_{p}^{-}(x)|v^{\star }|^{\hat{%
\beta}+1}|u_{n}|^{\hat{\alpha}+1}dx,
\end{equation*}
we distinguish two cases regarding exponents $\hat{\alpha}$ and $\hat{\beta}$%
.

\underline{\textbf{Case 1:}} $\dfrac{\hat{\alpha}+1}{p}+\dfrac{\hat{\beta}+1%
}{q}\geq 1.$

From condition (\ref{12}) there exists a pair $(\theta _{p},\theta _{q})\in
]1,p^{\star }[\times ]1,q^{\star }[$ such that 
\begin{equation*}
\dfrac{\hat{\alpha}+1}{p\theta _{p}}+\dfrac{\hat{\beta}+1}{q\theta _{q}}=1.
\end{equation*}%
By (\ref{7}), since the embedding $W_{0}^{1,p}(\Omega )\hookrightarrow
L^{p\theta _{p}}(\Omega )$ is compact, one gets 
\begin{equation*}
u_{n}\rightarrow u\text{ in }L^{p\theta _{p}}(\Omega ).
\end{equation*}
Fatou's Lemma implies 
\begin{equation}
\Vert ({a_{p}^{+})}^{1/\hat{\alpha}+1}|v^{\star }|^{\hat{\beta}+1/\hat{\alpha%
}+1}u\Vert _{{\hat{\alpha}+1}}\leq {\liminf_{n\rightarrow +\infty }}\Vert ({%
a_{p}^{+})}^{1/\hat{\alpha}+1}|v^{\star }|^{\hat{\beta}+1/\hat{\alpha}%
+1}u_{n}\Vert _{{\hat{\alpha}+1}}.  \label{aplus}
\end{equation}%
Since $a_{p}\in L^{\infty }(\Omega )$, by triangular and Young's
inequalities, we obtain 
\begin{equation*}
\begin{array}{l}
\Vert ({a_{p}^{+})}^{1/\hat{\alpha}+1}|v^{\star }|^{1/\hat{\alpha}%
+1}u_{n}\Vert _{\hat{\alpha}+1} \\ 
\leq \Vert ({a_{p}^{+})}^{1/\hat{\alpha}+1}|v^{\star }|^{\hat{\beta}+1/\hat{%
\alpha}+1}(u_{n}-u)\Vert _{\hat{\alpha}+1}+\Vert ({a_{p}^{+})}^{1/\hat{\alpha%
}+1}|v^{\star }|^{\hat{\beta}+1/\hat{\alpha}+1}u\Vert _{\hat{\alpha}+1} \\ 
\leq \Vert {a_{p}^{+}}\Vert _{\infty }^{1/\hat{\alpha}+1}\Vert v^{\star
}\Vert _{q\theta _{q}}^{\hat{\beta}+1/{\hat{\alpha}+1}}\Vert u_{n}-u\Vert
_{p\theta _{p}}+\Vert ({a_{p}^{+})}^{1/\hat{\alpha}+1}|v^{\star }|^{\hat{%
\beta}+1/\hat{\alpha}+1}u\Vert _{\hat{\alpha}+1}.%
\end{array}%
\end{equation*}%
Thus, due to (\ref{7}), one derives 
\begin{equation*}
\Vert {a_{p}^{+}}\Vert _{\infty }^{1/\hat{\alpha}+1}\Vert v^{\star }\Vert
_{q\theta _{q}}^{\hat{\beta}+1/{\hat{\alpha}+1}}\Vert u_{n}-u\Vert _{p\theta
_{p}}\rightarrow 0.
\end{equation*}%
Passing on the upper limit, it follows that 
\begin{equation}
{\limsup_{n\rightarrow +\infty }}\Vert ({a_{p}^{+})}^{1/\hat{\alpha}%
+1}|v^{\star }|^{\hat{\beta}+1/\hat{\alpha}+1}u_{n}\Vert _{{\hat{\alpha}+1}%
}\leq \Vert {a_{p}^{+}}^{1/\hat{\alpha}+1}|v^{\star }|^{\hat{\beta}+1/\hat{%
\alpha}+1}u\Vert _{{\hat{\alpha}+1}}.  \label{unu}
\end{equation}%
Hence, (\ref{aplus}) and (\ref{unu}) result in 
\begin{equation*}
{\lim_{n\rightarrow +\infty }}\Vert {a_{p}^{+}}^{1/\hat{\alpha}+1}|v^{\star
}|^{\hat{\beta}+1/\hat{\alpha}+1}u_{n}\Vert _{{\hat{\alpha}+1}}=\Vert {%
a_{p}^{+}}^{1/\hat{\alpha}+1}|v^{\star }|^{\hat{\beta}+1/\hat{\alpha}%
+1}u\Vert _{{\hat{\alpha}+1}}.
\end{equation*}

\underline{\textbf{Case 2:}} $0<\dfrac{\hat{\alpha}+1}{p}+\dfrac{\hat{\beta}%
+1}{q}<1.$

Observe that the argument used in the first case remains valid. Thus 
\begin{equation*}
\Vert ({a_{p}^{+})}^{1/\hat{\alpha}+1}|v^{\star }|^{\hat{\beta}+1/\hat{\alpha%
}+1}(u_{n}-u)\Vert _{\hat{\alpha}+1}\leq \Vert {a_{p}^{+}}\Vert _{\hat{r}%
}^{1/\hat{\alpha}+1}\Vert v^{\star }\Vert _{q}^{\hat{\beta}+1/\hat{\alpha}%
+1}\Vert u_{n}-u\Vert _{p},
\end{equation*}%
where ${\hat{r}}=(\dfrac{\hat{\alpha}+1}{p}+\dfrac{\hat{\beta}+1}{q})^{-1}$.
Moreover, considering the term 
\begin{equation*}
{\int_{\Omega }}a_{p}^{-}(x)|v^{\star }|^{\hat{\beta}+1}|u_{n}|^{\hat{\alpha}%
+1}dx,
\end{equation*}
a quite similar reasoning as above provides%
\begin{equation*}
{\lim_{n\rightarrow +\infty }}\Vert {a_{p}^{-}}^{1/\hat{\alpha}+1}|v^{\star
}|^{\hat{\beta}+1/\hat{\alpha}+1}u\Vert _{{\hat{\alpha}+1}}=\Vert {a_{p}^{-}}%
^{1/\hat{\alpha}+1}|v^{\star }|^{\hat{\beta}+1/\hat{\alpha}+1}u\Vert _{{\hat{%
\alpha}+1}}.
\end{equation*}%
Thereby, in both cases, we have ${\lim_{n\rightarrow +\infty }}%
A_{p}(u_{n})=A_{p}(u),$ which proves the claim.

Now, we prove that ${\lim_{n\rightarrow +\infty }}B_{p}(u_{n})=B_{p}(u)$.
Write $b_{p}=b_{p}^{+}-b_{p}^{-}$ and proceeding as in the first case, we
obtain on the one hand (the result remains the same if we change $b^{+}$ by $%
b^{-}$) 
\begin{equation}
\Vert ({b_{p}^{+})}^{1/p}|u|\Vert _{p}\leq {\liminf_{n\rightarrow +\infty }}%
\Vert ({b_{p}^{+})}^{1/p}|u_{n}|\Vert _{p}.  \label{binf}
\end{equation}%
and on the other hand 
\begin{equation}
\Vert ({b_{p}^{+})}^{1/p}|u_{n}|\Vert _{p}\leq \Vert ({b_{p}^{+})}%
^{1/p}|u_{n}-u|\Vert _{p}+\Vert ({b_{p}^{+})}^{1/p}|u|\Vert _{p}.  \label{b+}
\end{equation}%
H\"{o}lder's inequality implies 
\begin{equation}
{\int_{\Omega }}|b_{p}^{+}||u_{n}-u|^{p}dx\leq \left( {\int_{\Omega }}%
|b_{p}^{+}|^{\delta _{p}}dx\right) ^{1/\delta _{p}}\left( {\int_{\Omega }}%
|u_{n}-u|^{\delta _{p}^{\prime }p}dx\right) ^{1/{\delta _{p}^{\prime }}}.
\label{b}
\end{equation}%
Since $\delta _{p}>N/p,$ the embedding $W_{0}^{1,p}(\Omega )\hookrightarrow
L^{\delta _{p}^{\prime }p}(\Omega )$ is compact (here $\delta _{p}^{\prime }=%
\dfrac{\delta _{p}}{\delta _{p}-1}$). Thereby, the sequence $u_{n}$
converges strongly to $u$ in $L^{\delta _{p}^{\prime }p}(\Omega )$ and
therefore, the right hand in (\ref{b}) tends to $0.$ Thus, by (\ref{b+}),
one has 
\begin{equation*}
{\limsup_{n\rightarrow +\infty }}\Vert ({b_{p}^{+})}^{1/p}|u_{n}|\Vert
_{p}=\Vert ({b_{p}^{+})}^{1/p}|u|\Vert _{p}.
\end{equation*}%
Combining with (\ref{binf}) it follows clearly that 
\begin{equation*}
{\int_{\Omega }}b_{p}^{+}|u_{n}|^{p}dx\rightarrow {\int_{\Omega }}%
b_{p}^{+}|u|^{p}dx.
\end{equation*}%
Similarly, taking $b_{p}^{-}$ instead of $b_{p}^{+},$ one has 
\begin{equation*}
{\int_{\Omega }}b_{p}^{-}(x)|u_{n}|^{p}dx\rightarrow {\int_{\Omega }}%
b_{p}^{-}(x)|u|^{p}dx.
\end{equation*}%
Consequently, the application $B_{p}$ is weakly continuous on $%
W_{0}^{1,p}(\Omega ).$ The proof is achieved.
\end{proof}

\begin{lemma}
\label{L4}Let $\hat{\alpha}\neq {p-1,}$ $a_{p}^{+}\neq 0$ in $\Omega $ and
assume $\lambda \geq 0$ such that one of the following conditions is
satisfied: either

\begin{description}
\item[$($\textrm{i}$)$] $\lambda <\lambda _{b_{p}}$ $;$ or

\item[$($\textrm{ii}$)$] $\lambda =\lambda _{b_{p}}$ or $\lambda
_{b_{p}}<\lambda <\lambda _{b_{p}}+\varepsilon _{p}$ for a certain $%
\varepsilon _{p}>0$, and $\int_{\Omega }a_{p}u_{p}^{\hat{\alpha}+1}|v^{\star
}|^{\hat{\beta}+1}dx<0,$ where $u_{p}$ is the eigenvalue associated to the
first eigenvalue $\lambda _{b_{p}}.$
\end{description}

Then problem 
\begin{equation}
\left\{ 
\begin{array}{ll}
-\Delta _{p}z-{a_{p}}(x)z|z|^{\hat{\alpha}-1}|v^{\star }|^{\hat{\beta}%
+1}-\lambda {b_{p}}(x)z|z|^{p-2}=0 & \hbox{ in }\Omega \\ 
z=0 & \hbox{ on }\partial \Omega%
\end{array}%
\right.  \label{sstem}
\end{equation}%
admits at least one weak positive solution $\mathcal{U}$ in $%
W_{0}^{1,p}(\Omega )$.
\end{lemma}

\begin{proof}
Inspired by \cite[sections 3.3 - 3.6]{p}, let consider the Euler functional 
\begin{equation*}
E_{p,\lambda }(u)=\dfrac{1}{p}{\int_{\Omega }}\left\vert \nabla u\right\vert
^{p}dx-\dfrac{\lambda }{p}B_{p}(u)-\dfrac{1}{\hat{\alpha}+1}A_{p}(u),
\end{equation*}%
where the applications $A_{p}$ and $B_{p}$ are defined in (\ref{ABu}).%
\newline
By Definition \ref{ccpoint} and under the conditions $t\in \mathbb{R}$ and $%
\displaystyle\int_{\Omega }|\nabla z|^{p}dx-\lambda {B_{p}(z)}=1,$ one gets 
\begin{equation*}
\tilde{E}_{p}(t,z)=E_{p}(tz)=\dfrac{t^{p}}{p}-\dfrac{t^{\hat{\alpha}+1}}{%
\hat{\alpha}+1}A_{p}(z).
\end{equation*}%
$E_{p,\lambda }(u)$ becomes 
\begin{equation*}
{\hat{E}}_{p,\lambda }(z)=\displaystyle\max_{t\in \mathbb{R}\setminus \{0\}}%
\tilde{E}_{p}(t,z)=\left( \dfrac{1}{p}-\dfrac{1}{\hat{\alpha}+1}\right)
A_{p}^{-\frac{p}{(\hat{\alpha}+1)-p}}(z).
\end{equation*}%
In addition, the assumption $\hat{\alpha}+1\neq {p}$ ensures from Definition %
\ref{ccpoint} that the conditional critical point of ${\hat{E}}_{p,\lambda }$
is related to the maximization problem 
\begin{equation*}
0<M_{\lambda }={\sup_{z\in W_{0}^{1,p}(\Omega )}}\left\{ A_{p}(z):\,\Vert
z\Vert _{1,p}^{p}-\lambda {B_{p}}(z)=1\right\} .
\end{equation*}%
Thanks to Lemma \ref{AB}, it is clear that the required assumptions
(f0)-(g0) in \cite{p} or (AO)-(BO) in \cite{ch} are fulfilled. Consequently,
by $($\textrm{i}$)$ or $($\textrm{ii}$)$ in Lemma \ref{L4}, Proposition \ref%
{P1} ensures that problem (\ref{sstem}) admits at least one positive weak
solution $\mathcal{U}$.
\end{proof}

\bigskip

Now, we are ready to prove the positivity result stated in Theorem \ref{T2}.

\begin{proof}[\textbf{Proof of Theorem \protect\ref{T2}}]
Let us define on $\Omega \times W_{0}^{1,p}(\Omega )$ the operator $\mathcal{%
L}_{p,v^{\star }}$ as follows 
\begin{equation*}
\mathcal{L}_{p,v^{\star }}(x,z)=-\Delta _{p}z-f(x,z,v^{\star })\text{ \ for
every }z\in W_{0}^{1,p}(\Omega ).
\end{equation*}%
By $($\textrm{H.4}$)$ one has 
\begin{equation}
\mathcal{L}_{p,v^{\star }}(x,\mathcal{U})\leq -\Delta _{p}\mathcal{U}-{a_{p}}%
(x)\mathcal{U}|\mathcal{U}|^{\hat{\alpha}-1}|v^{\star }|^{\hat{\beta}+1}-{%
b_{p}}(x)\mathcal{U}|\mathcal{U}|^{p-2}=0\,\,\,\,\hbox{ in }\Omega .
\label{LU}
\end{equation}%
In addition, using again $($\textrm{H.4}$)$, we have 
\begin{equation}
\mathcal{L}_{p,v^{\star }}(x,\mathcal{U})\leq 0\leq \mathcal{L}_{p,v^{\star
}}(x,u^{\star }).  \label{L}
\end{equation}%
Here, (\ref{LU}) and (\ref{L}) should be understood in the weak sense, that
is, 
\begin{equation*}
\begin{array}{l}
{\int_{\Omega }}|\nabla \mathcal{U}|^{p-2}\nabla \mathcal{U}\nabla \phi dx-{%
\int_{\Omega }}f(x,\mathcal{U},v^{\star })\phi dx \\ 
\leq {\int_{\Omega }}|\nabla \mathcal{U}|^{p-2}\nabla \mathcal{U}\nabla \phi
dx-{\int_{\Omega }}b_{p}(x)\mathcal{U}\phi |\mathcal{U}|^{p-2}dx-{%
\int_{\Omega }}{a_{p}}(x)\mathcal{U}\phi |\mathcal{U}|^{\hat{\alpha}%
-1}|v^{\star }|^{\hat{\beta}+1}dx%
\end{array}%
\end{equation*}%
and 
\begin{equation*}
\begin{array}{l}
{\int_{\Omega }}|\nabla \mathcal{U}|^{p-2}\nabla \mathcal{U}\nabla \phi dx-{%
\int_{\Omega }}f(x,\mathcal{U},v^{\star })\phi dx \\ 
\leq {\int_{\Omega }}|\nabla u^{\star }|^{p-2}\nabla u^{\star }\nabla \phi
dx-{\int_{\Omega }}f(x,u^{\star },v^{\star })\phi dx%
\end{array}%
\end{equation*}%
for all $\phi \in W_{0}^{1,p}(\Omega ),$ $\phi \geq 0$ in $\Omega .$

We claim that $\mathcal{U}\leq u^{\star }$ in $\Omega .$ Indeed, testing
with $(\mathcal{U}-u^{\star })^{+}$ the equation (\ref{L}) and integrating
over $\Omega $, one has 
\begin{equation*}
\langle -\Delta _{p}\mathcal{U}-\left( -\Delta _{p}u^{\star }\right) ,(%
\mathcal{U}-u^{\star })^{+}\rangle _{-1,1}\leq {\int_{\Omega }}\left( f(x,%
\mathcal{U},v^{\star })-f(x,u^{\star },v^{\star })\right) (\mathcal{U}%
-u^{\star })^{+}dx,
\end{equation*}%
which is equivalent to 
\begin{equation*}
\langle -\Delta _{p}\mathcal{U}-\left( -\Delta _{p}u^{\star }\right) ,(%
\mathcal{U}-u^{\star })^{+}\rangle _{-1,1}\leq {\int_{\{u^{\star }\leq 
\mathcal{U}\}}}\left( f(x,\mathcal{U},v^{\star })-f(x,u^{\star },v^{\star
})\right) (\mathcal{U}-u^{\star })^{+}dx.
\end{equation*}%
Since 
\begin{equation*}
\Vert (\mathcal{U}-u^{\star })^{+}\Vert _{p}^{p}\leq C_{p}\Vert (\mathcal{U}%
-u^{\star })^{+}\Vert _{1,p}^{p}\leq \langle -\Delta _{p}\mathcal{U}-\left(
-\Delta _{p}u^{\star }\right) ,(\mathcal{U}-u^{\star })^{+}\rangle _{-1,1},
\end{equation*}%
the monotonicity assumption $($\textrm{H.3}$)$ implies that the right hand
side remains negative while the left hand side is positive. Thus $(\mathcal{U%
}-u^{\star })^{+}=0$ in $\Omega $ forces $\mathcal{U}\leq u^{\star }$ in $%
\Omega $.

Finally, because $\mathcal{U}>0$ in $\Omega ,$ we infer that $u^{\star }>0$
in $\Omega .$ Analogously, we derive that $v^{\star }>0$ in $\Omega $. The
proof is complete.
\end{proof}

\section{Boundedness}

\label{S3}

\begin{lemma}
\label{approx} Assume $($\textrm{H.1}$)$ holds. Then, for any solution $%
(u^{\star },v^{\star })$ of system (\ref{p}), there exists a sequence $%
(u_{\varepsilon },v_{\varepsilon })\in \left( C^{1}(\bar{\Omega})\cap C^{2}({%
\Omega })\right) ^{2}$ such that%
\begin{equation*}
(u_{\varepsilon },v_{\varepsilon })\rightarrow (u^{\star },v^{\star })\text{
strongly in }W_{0}^{1,p}(\Omega )\times W_{0}^{1,q}(\Omega ).
\end{equation*}
\end{lemma}

\begin{proof}
By $($\textrm{H.1}$)$, $f$ and $g$ belong in $L^{p_{C}^{\prime }}(\Omega )$
and $L^{q_{C}^{\prime }}(\Omega ),$ respectively. Then, since $C_{0}^{\infty
}(\Omega )$ is dense in $L^{p_{C}^{\prime }}(\Omega )$ and $L^{q_{C}^{\prime
}}(\Omega ),$ one can find a pair $(f_{\varepsilon },g_{\varepsilon })\in
C_{0}^{\infty }(\Omega )\times C_{0}^{\infty }(\Omega )$ such that 
\begin{equation*}
\Vert f_{\varepsilon }-f\Vert _{p_{C}^{\prime }}\rightarrow 0\text{ \ and \ }%
\Vert g_{\varepsilon }-g\Vert _{p_{C}^{\prime }}\rightarrow 0.
\end{equation*}
Therefore, $f_{\varepsilon }$ admits a subsequence $f_{\varepsilon _{n}}$
which converges \textit{a.e} in $\Omega .$ Thus, there exists a constant $%
c_{0}>0$, independent of $\varepsilon _{n},$ such that $\Vert f_{\varepsilon
_{n}}\Vert _{\infty }\leq c_{0}$ (see, e.g., \cite{o}).

Let $(u_{\varepsilon _{n}},v_{\varepsilon _{n}})$ be a solution defined as
follows:%
\begin{equation}
\left\{ 
\begin{array}{rl}
-\Delta _{p}^{\varepsilon }u_{\varepsilon _{n}}=f_{\varepsilon _{n}} & %
\hbox{ in }\Omega \\ 
-\Delta _{q}^{\varepsilon }v_{\varepsilon _{n}}=g_{\varepsilon _{n}} & %
\hbox{ in }\Omega \\ 
u_{\varepsilon _{n}}=v_{\varepsilon _{n}}=0 & \hbox{ on }\partial \Omega ,%
\end{array}%
\right.  \label{deltaeps}
\end{equation}%
where $\Delta _{p}^{\varepsilon _{n}}u$ is given by 
\begin{equation*}
{div}\left[ \left( |\nabla u|^{2}+{\varepsilon _{n}}\right) ^{\frac{1}{2}%
(p-2)}\nabla u\right] ,\text{ for all }\varepsilon >0,\text{ and all }p>1.
\end{equation*}

It is well known that $(u_{\varepsilon _{n}},v_{\varepsilon _{n}})\in \left(
C^{1}(\bar{\Omega})\cap C^{2}({\Omega })\right) ^{2}$ (see \cite{gt}).
Thereby,%
\begin{equation}
{\int_{\Omega }}\left( |\nabla u_{\varepsilon _{n}}|^{2}+{\varepsilon _{n}}%
\right) ^{p-2/2}|\nabla u_{\varepsilon _{n}}|^{2}dx={\int_{\Omega }}%
f_{\varepsilon _{n}}u_{\varepsilon _{n}}dx  \label{22}
\end{equation}%
and%
\begin{equation*}
{\int_{\Omega }}\left( |\nabla v_{\varepsilon _{n}}|^{2}+{\varepsilon _{n}}%
\right) ^{q-2/2}|\nabla v_{\varepsilon _{n}}|^{2}dx={\int_{\Omega }}%
g_{\varepsilon _{n}}v_{\varepsilon _{n}}dx.
\end{equation*}%
Applying the H\"{o}lder's inequality in the right-hand side of (\ref{22}),
the below estimate occurs%
\begin{equation}
\left\vert {\int_{\Omega }}f_{\varepsilon _{n}}u_{\varepsilon
_{n}}dx\right\vert \leq \Vert f_{\varepsilon _{n}}\Vert _{p_{C}^{\prime
}}\Vert u_{\varepsilon _{n}}\Vert _{p}\leq \Vert f\Vert _{p_{C}^{\prime
}}\Vert u_{\varepsilon _{n}}\Vert _{p}\leq C_{p}\Vert f\Vert _{p_{C}^{\prime
}}\Vert u_{\varepsilon _{n}}\Vert _{1,p}.  \label{rhd}
\end{equation}%
Now, we deal with the left-hand side. For $p>1$, we claim that 
\begin{equation}
{\int_{\Omega }}|\nabla u_{\varepsilon _{n}}|^{p}dx-{\varepsilon _{n}}%
^{p/2}meas(\Omega )\leq {\int_{\Omega }}\left( |\nabla u_{\varepsilon
_{n}}|^{2}+{\varepsilon _{n}}\right) ^{p-2/2}|\nabla u_{\varepsilon
_{n}}|^{2}dx.  \label{fivpfi0}
\end{equation}%
Indeed, by elementary algebra inequality one has 
\begin{equation*}
{\int_{\Omega }}|\nabla u_{\varepsilon _{n}}|^{p}dx\leq \left\{ 
\begin{array}{ll}
{\int_{\Omega }}\left( |\nabla u_{\varepsilon _{n}}|^{2}+{\varepsilon _{n}}%
\right) ^{p-2/2}|\nabla u_{\varepsilon _{n}}|^{2}dx & \text{ if }2\leq p, \\ 
{\int_{\Omega }}(-\Delta _{p}u_{\varepsilon _{n}})u_{\varepsilon _{n}}dx+{%
\varepsilon _{n}}^{p/2}meas(\Omega ) & \text{ if }p<2.%
\end{array}%
\right. 
\end{equation*}%
From (\ref{rhd}) and (\ref{fivpfi0}), we deduce 
\begin{equation*}
\begin{array}{l}
{\int_{\Omega }}|\nabla u_{\varepsilon _{n}}|^{p}dx-{\varepsilon _{n}}%
^{p/2}meas(\Omega )\leq C_{p}\Vert f\Vert _{p_{C}^{\prime }}\Vert
u_{\varepsilon _{n}}\Vert _{1,p} \\ 
\leq \dfrac{1}{p^{\prime }}(C_{p}\Vert f\Vert _{p_{C}^{\prime }})^{p^{\prime
}}+\dfrac{1}{p}{\int_{\Omega }}|\nabla u_{\varepsilon _{n}}|^{p}dx%
\end{array}%
\end{equation*}%
or again, 
\begin{equation*}
\dfrac{1}{p^{\prime }}{\int_{\Omega }}|\nabla u_{\varepsilon
_{n}}|^{p}dx\leq meas(\Omega )+\dfrac{1}{p^{\prime }}(C_{p}\Vert f\Vert
_{p_{C}^{\prime }})^{p^{\prime }}.
\end{equation*}%
Thus, it's readily seen that $u_{\varepsilon _{n}}$ is bounded in $%
W_{0}^{1,p}(\Omega ).$ Following the same agrument we obtain that $%
v_{\varepsilon _{n}}$ is bounded in $W_{0}^{1,q}(\Omega )$. Let $(\tilde{u},%
\tilde{v})$ be the weak limit of the sequence $(u_{\varepsilon
_{n}},v_{\varepsilon _{n}})$ in $W_{0}^{1,p}(\Omega )\times
W_{0}^{1,q}(\Omega )$. The proof is completed by showing that $(\tilde{u},%
\tilde{v})=(u^{\star },v^{\ast }).$ To this end, let us first show that the
strong convergence $u_{\varepsilon _{n}}\rightarrow \tilde{u}$ holds true.
Obviously, by weak semi-continuous arguments, it is well known that 
\begin{equation*}
{\int_{\Omega }}|\nabla \tilde{u}|^{p}dx\leq {\liminf_{{\varepsilon _{n}}%
\rightarrow 0}}{\int_{\Omega }}|\nabla u_{\varepsilon _{n}}|^{p}dx.
\end{equation*}%
Moreover, consider the application 
\begin{equation*}
\Phi _{\varepsilon }:z\longmapsto {\int_{\Omega }}\left( |\nabla
z|^{2}+\varepsilon \right) ^{p/2}dx
\end{equation*}%
and let $\partial \Phi _{\varepsilon }(z)$ be its subdifferential set.
Clearly, $\partial \Phi _{\varepsilon }(z)$ is reduced to a single element $%
-\Delta _{p}^{\varepsilon }(z)$ which is defined on $W_{0}^{1,p}(\Omega )$
as follows 
\begin{equation*}
-\Delta _{p}^{\varepsilon }(z):h\longmapsto {\int_{\Omega }}\left( |\nabla
z|^{2}+\varepsilon \right) ^{p-2/2}\nabla z\nabla h\text{ }dx.
\end{equation*}%
For all $z\in W_{0}^{1,p}(\Omega ),$ we have 
\begin{equation}
\Phi _{\varepsilon }(z)-\Phi _{\varepsilon }(u_{\varepsilon _{n}})\geq {%
\int_{\Omega }}f_{\varepsilon _{n}}(z-u_{\varepsilon _{n}})dx.  \label{fivp}
\end{equation}%
In particular, setting $z=\tilde{u}$ and using (\ref{fivpfi0}), it follows
that 
\begin{equation*}
{\int_{\Omega }}|\nabla u_{\varepsilon _{n}}|^{p}dx\leq {\int_{\Omega }}%
f_{\varepsilon _{n}}(u_{\varepsilon _{n}}-\tilde{u})dx+{\varepsilon _{n}}%
^{p/2}meas(\Omega )+\Phi _{\varepsilon _{n}}(\tilde{u}).
\end{equation*}%
Passing to the upper-limit on ${\varepsilon }{_{n}},$ we get 
\begin{equation*}
{\limsup_{{\varepsilon _{n}}\rightarrow 0}}{\int_{\Omega }}|\nabla
u_{\varepsilon _{n}}|^{p}dx\leq {\int_{\Omega }}|\nabla \tilde{u}|^{p}dx.
\end{equation*}%
Thus, it follows that 
\begin{equation}
{\int_{\Omega }}|\nabla \tilde{u}|^{p}dx={\lim_{{\varepsilon _{n}}%
\rightarrow 0}}{\int_{\Omega }}|\nabla u_{\varepsilon _{n}}|^{p}dx.
\label{uvptildeu}
\end{equation}%
Recalling that $u_{\varepsilon _{n}}\rightharpoonup \tilde{u}$ weakly in $%
W_{0}^{1,p}(\Omega ),$ since $W_{0}^{1,p}(\Omega )$ is an uniform convex
Banach space, we conclude that%
\begin{equation*}
u_{\varepsilon _{n}}\rightarrow \tilde{u}\text{ \ in }W_{0}^{1,p}(\Omega ).
\end{equation*}

Now, we are ready to show that $\tilde{u}=u^{\star }.$ Set $\Phi _{0}(z)={%
\int_{\Omega }}|\nabla z|^{p}dx$ and denote by $\partial \Phi _{0}(z)$ its
subdifferential set. Combining (\ref{fivpfi0}), (\ref{fivp}) and passing to
the limit on ${\varepsilon }{_{n}},$ it follows from (\ref{uvptildeu}) that%
\begin{equation*}
\Phi _{0}(z)-\Phi _{0}(\tilde{u})\geq {\int_{\Omega }}f(x,u^{\star
},v^{\star })(z-\tilde{u})dx,\,\,\hbox{ for  all }z\in W_{0}^{1,p}(\Omega ).
\end{equation*}%
Since $\partial \Phi _{0}(z)$ contains a single value $-\Delta _{p}z,$ we
derive that $f(x,u^{\star },v^{\star })\in \partial \Phi _{0}(\tilde{u})$
and therefore $-\Delta _{p}\tilde{u}=f(x,u^{\star },v^{\star }).$ However,
the definition of $u^{\star }$ (see Theorem \ref{T1}) leads to $-\Delta _{p}%
\tilde{u}=-\Delta _{p}u^{\star }$ in $\Omega .$ By weak comparison
principle, this implies $\tilde{u}=u^{\star }$ in $\Omega $. A quite similar
argument produces that $\tilde{v}=u^{\star }$ in $\Omega $, ending the proof.
\end{proof}

\mathstrut 

The next part is devoted to establish the boundedness of the solution $%
(u^{\star },v^{\star })$.

\begin{lemma}
\label{uepsnvepsn} For all $k\in 
\mathbb{N}
,$ let $(\delta _{k})$ and $(\gamma _{k})$ be the sequences 
\begin{equation}
\delta _{k}=pCf_{k},\qquad \gamma _{k}=qCf_{k},  \label{flbmuk}
\end{equation}%
where 
\begin{equation}
f_{k}=D\left( C^{k}+\dfrac{1}{D}\right) ,\,\text{with \ }0<D<{\min }\left( 
\dfrac{p^{\star }}{pC},\dfrac{q^{\star }}{qC}\right) -1.  \label{fCD}
\end{equation}%
Then, the pair $(u_{\varepsilon _{n}},v_{\varepsilon _{n}})$ defined as in (%
\ref{deltaeps}) is bounded in $L^{\delta _{k}}(\Omega )\times L^{\gamma
_{k}}(\Omega ),$ for all $k\geq 1.$
\end{lemma}

\begin{proof}
The Lemma is proved if we show that the sequence $(u_{\varepsilon
_{n}},v_{\varepsilon _{n}})$ follows the iterative scheme 
\begin{equation}
\hbox{ If }(u_{\varepsilon _{n}},v_{\varepsilon _{n}})\in L^{\delta
_{k}}(\Omega )\times L^{\gamma _{k}}(\Omega )\hbox{ then }(u_{\varepsilon
_{n}},v_{\varepsilon _{n}})\in L^{\delta _{k+1}}(\Omega )\times L^{\gamma
_{k+1}}(\Omega ).  \label{15}
\end{equation}%
\underline{\textbf{Step 1}}\textbf{: }We prove that $(u_{\varepsilon
_{n}},v_{\varepsilon _{n}})$ satisfies (\ref{15}) for $k=0$.

Combining (\ref{flbmuk}) and (\ref{fCD}) one has 
\begin{equation*}
\delta _{0}=pC(D+1)<p^{\star },\,\,\gamma _{0}=qC(D+1)<q^{\star }.
\end{equation*}%
Consequently, the embeddings $W_{0}^{1,p}(\Omega )\hookrightarrow L^{\delta
_{0}}(\Omega )$ and $W_{0}^{1,q}(\Omega )\hookrightarrow L^{\gamma
_{0}}(\Omega )$ are continuous, leading to $u_{\varepsilon _{n}}\in
L^{\delta _{0}}(\Omega )$ and $v_{\varepsilon _{n}}\in L^{\gamma
_{0}}(\Omega ).$

\underline{\textbf{Step 2}}\textbf{: }Let us prove that if $(u_{\varepsilon
_{n}},v_{\varepsilon _{n}})\in L^{\delta _{l}}(\Omega )\times L^{\gamma
_{l}}(\Omega )$ for $l\in 
\mathbb{N}
,$ $l\leq k,$ then $(u_{\varepsilon _{n}},v_{\varepsilon _{n}})\in L^{\delta
_{k+1}}(\Omega )\times L^{\gamma _{k+1}}(\Omega ).$

For $k\in \mathbb{N},$ we define the sequences $a_{k}$ and $b_{k}$ by 
\begin{equation*}
a_{k}=DC^{k+1}{p},\,\,\,b_{k}=DC^{k+1}{q},
\end{equation*}%
where constants $C$ and $D$ verify (\ref{fCD}). Testing the first and the
second equations in (\ref{deltaeps}) with $u_{\varepsilon
_{n}}|u_{\varepsilon _{n}}|^{a_{k}}$ and $v_{\varepsilon
_{n}}|v_{\varepsilon _{n}}|^{b_{k}}$, respectively, integrating over $\Omega 
$, we get 
\begin{equation}
{\int_{\Omega }}\left( |\nabla u_{\varepsilon _{n}}|^{2}+{\varepsilon }%
\right) ^{p-2/2}\nabla u_{\varepsilon _{n}}\nabla (u_{\varepsilon
_{n}}|u_{\varepsilon _{n}}|^{a_{k}})dx={\int_{\Omega }}u_{\varepsilon
_{n}}|u_{\varepsilon _{n}}|^{a_{k}}f_{\varepsilon _{n}}dx,  \label{ua}
\end{equation}%
and%
\begin{equation}
{\int_{\Omega }}\left( |\nabla v_{\varepsilon _{n}}|^{2}+{\varepsilon }%
\right) ^{q-2/2}\nabla v_{\varepsilon _{n}}\nabla (v_{\varepsilon
_{n}}|v_{\varepsilon _{n}}|^{b_{k}})dx={\int_{\Omega }}v_{\varepsilon
_{n}}|v_{\varepsilon _{n}}|^{b_{k}}g_{\varepsilon _{n}}dx.  \label{vb}
\end{equation}%
Clearly, for all $p>1,$ it holds%
\begin{equation}
\begin{array}{l}
{\int_{\Omega }}\left( |\nabla u_{\varepsilon _{n}}|^{2}+{\varepsilon _{n}}%
\right) ^{p-2/2}\left\vert \nabla u_{\varepsilon _{n}}\right\vert
^{2}|u_{\varepsilon _{n}}|^{a_{k}}dx \\ 
\geq {\int_{\Omega }}|\nabla u_{\varepsilon _{n}}|^{p}|u_{\varepsilon
_{n}}|^{a_{k}}dx-{\varepsilon }^{p/2}{\int_{\Omega }}|u_{\varepsilon
_{n}}|^{a_{k}}dx \\ 
\geq \dfrac{{\int_{\Omega }}\left\vert \nabla |u_{\varepsilon _{n}}|^{1+%
\frac{a_{k}}{p}}\right\vert ^{p}dx}{(1+\frac{a_{k}}{p})^{p}}-{\varepsilon }%
^{p/2}{\int_{\Omega }}|u_{\varepsilon _{n}}|^{a_{k}}dx.%
\end{array}
\label{19}
\end{equation}%
Thus, as in (\ref{fivpfi0}), the left-hand side in (\ref{ua}) is estimated
by 
\begin{equation*}
\begin{array}{l}
{\int_{\Omega }}\left( |\nabla u_{\varepsilon _{n}}|^{2}+{\varepsilon _{n}}%
\right) ^{p-2/2}\nabla u_{\varepsilon _{n}}\nabla (u_{\varepsilon
_{n}}|u_{\varepsilon _{n}}|^{a_{k}})dx \\ 
=(a_{k}+1){\int_{\Omega }}\left( |\nabla u_{\varepsilon _{n}}|^{2}+{%
\varepsilon _{n}}\right) ^{p-2/2}|\nabla u_{\varepsilon
_{n}}|^{2}|u_{\varepsilon _{n}}|^{a_{k}}dx \\ 
\geq \dfrac{a_{k}+1}{\left( 1+\frac{a_{k}}{p}\right) ^{p}}{\int_{\Omega }}%
\left\vert \nabla |u_{\varepsilon _{n}}|^{1+\frac{a_{k}}{p}}\right\vert
^{p}dx-(a_{k}+1){\varepsilon }^{p/2}{\int_{\Omega }}|u_{\varepsilon
_{n}}|^{a_{k}}dx.%
\end{array}%
\end{equation*}%
Moreover, since the sequence $u_{\varepsilon _{n}}$ belongs to $C^{1}(%
\overline{\Omega }),$ $u_{\varepsilon _{n}}$ belongs to $W_{0}^{1,p}(\Omega
) $ and therefore $u_{\varepsilon _{n}}|u_{\varepsilon _{n}}|^{\frac{a_{k}}{p%
}} $ belongs to $W_{0}^{1,p}(\Omega ).$ Moreover, by $($\textrm{H.1}$)$, one
may write $1<pC<p^{\star }$ which ensures that the embedding $%
W_{0}^{1,p}(\Omega )\hookrightarrow L^{pC}(\Omega )$ is continuous. Hence,
there exists a constant $C_{pC}>0$ such that 
\begin{equation}
\left( {\int_{\Omega }}\left\vert u_{\varepsilon _{n}}\right\vert ^{{%
pC\left( 1+\frac{a_{k}}{p}\right) }}dx\right) ^{1/C}\leq C_{pC}^{p}{%
\int_{\Omega }}\left\vert \nabla u_{\varepsilon _{n}}{^{1+\frac{a_{k}}{p}}}%
\right\vert ^{p}dx.  \label{20}
\end{equation}%
However, since 
\begin{equation}
pC\left( 1+\frac{a_{k}}{p}\right) =pC(1+{D}C^{k+1})=\delta _{k+1},
\end{equation}%
the iterative inclusion $L^{pC\left( 1+\frac{a_{k}}{p}\right) }(\Omega
)\subset L^{\delta _{k+1}}(\Omega )$ holds true and then 
\begin{equation*}
\begin{array}{l}
\left( {\int_{\Omega }}\left\vert u_{\varepsilon _{n}}\right\vert ^{\delta
_{k+1}}dx\right) ^{1/\delta _{k+1}}\leq \left( vol\,\Omega \right) ^{\frac{1%
}{\delta _{k+1}}-\frac{1}{{pC\left( 1+\frac{a_{k}}{p}\right) }}}\left( {%
\int_{\Omega }}\left\vert u_{\varepsilon _{n}}\right\vert ^{{pC\left( 1+%
\frac{a_{k}}{p}\right) }}dx\right) ^{1/{pC\left( 1+\frac{a_{k}}{p}\right) }},%
\end{array}%
\end{equation*}%
or again, 
\begin{equation}
\begin{array}{l}
\left( vol\,\Omega \right) ^{\frac{1}{C}-\frac{1+C{D}C^{k}}{C(1+DC^{k+1})}%
}\left( {\int_{\Omega }}\left\vert u_{\varepsilon _{n}}\right\vert ^{\delta
_{k+1}}dx\right) ^{{p}\left( 1+\frac{a_{k}}{p}\right) /\delta _{k+1}}\leq
\left( {\int_{\Omega }}\left\vert u_{\varepsilon _{n}}\right\vert ^{{%
pC\left( 1+\frac{a_{k}}{p}\right) }}dx\right) ^{1/C}.%
\end{array}
\label{21}
\end{equation}%
Gathering (\ref{19}) - (\ref{21}) together\textbf{,} the estimate on the
left-hand side in (\ref{ua}) becomes 
\begin{equation*}
\begin{array}{l}
\dfrac{1+C{D}C^{k}}{C(1+{D}C^{k+1})C_{pC}}\left( vol\,\Omega \right) ^{\frac{%
1}{C}-\frac{1+C{D}C^{k}}{C(1+DC^{k+1})}}\left( {\int_{\Omega }}\left\vert
u_{\varepsilon _{n}}\right\vert ^{\delta _{k+1}}dx\right) ^{{p}\left( 1+%
\frac{a_{k}}{p}\right) /\delta _{k+1}} \\ 
\leq {\int_{\Omega }}|\nabla u_{\varepsilon _{n}}|^{p-2}\nabla
u_{\varepsilon _{n}}\nabla (u_{\varepsilon _{n}}|u_{\varepsilon
_{n}}|^{a_{k}})dx.%
\end{array}%
\end{equation*}%
Now, we focus on the right hand side of (\ref{ua}). First, we have 
\begin{equation*}
{\int_{\Omega }}\left\vert u_{{\varepsilon _{n}}}|u_{{\varepsilon _{n}}%
}|^{a_{k}}f_{\varepsilon _{n}}\right\vert dx\leq c_{0}\left\vert \Omega
\right\vert ^{1/r_{k}}\left( {\int_{\Omega }}|u_{\varepsilon _{n}}|^{\delta
_{k}}dx\right) ^{\frac{a_{k}+1}{\delta _{k}}}
\end{equation*}%
where 
\begin{equation*}
r_{k}=\dfrac{pCDC^{k}+pC}{pC-1}.
\end{equation*}%
Consequently, there is a constant $R_{k},$ depending on $k,$ such that 
\begin{equation*}
\begin{array}{rl}
\Vert u_{\varepsilon _{n}}\Vert _{\delta _{k+1}}^{a_{k}+p}\leq & R_{k}\Vert
u_{\varepsilon _{n}}\Vert _{\delta _{k}}^{a_{k}+1}+(a_{k}+1){\varepsilon _{n}%
}^{p/2}{\int_{\Omega }}|u_{\varepsilon _{n}}|^{a_{k}}dx \\ 
\leq & R_{k}\Vert u_{\varepsilon _{n}}\Vert _{\delta
_{k}}^{a_{k}+1}+(a_{k}+1){\varepsilon _{n}}^{p/2}|\Omega |^{1/t_{k}}\Vert
u_{\varepsilon _{n}}\Vert _{\delta _{k}}^{a_{k}},%
\end{array}%
\end{equation*}%
where $t_{k}=DC^{k}+1.$ This means that the inclusion $L^{\delta
_{k}}(\Omega )\subset L^{\delta _{k+1}}(\Omega )$ holds true for all $k\geq
1.$ Therefore, since the domain $\Omega $ is bounded, one gets 
\begin{equation*}
\Vert u_{\varepsilon _{n}}\Vert _{\delta _{k}}\leq |\Omega |^{1/\delta
_{k}-1/\delta _{k+1}}\Vert u_{\varepsilon _{n}}\Vert _{\delta _{k+1}}.
\end{equation*}%
Then 
\begin{equation*}
|\Omega |^{1/\delta _{k+1}-1/\delta _{k}}\Vert u_{\varepsilon _{n}}\Vert
_{\delta _{k}}^{a_{k}+p}\leq R_{k}\Vert u_{\varepsilon _{n}}\Vert _{\delta
_{k}}^{a_{k}+1}+(a_{k}+1)|\Omega |^{1/t_{k}}\Vert u_{\varepsilon _{n}}\Vert
_{\delta _{k}}^{a_{k}},
\end{equation*}%
showing that the sequence $u_{\varepsilon _{n}}$ is bounded in every
Lebesgue space $L^{\delta _{k}}(\Omega ),$ $k\geq 1$. This ends the proof.
\end{proof}

\mathstrut 

\begin{proof}[\textbf{Proof of Theorem \protect\ref{T2}}]
From Lemma \ref{approx}, along a relabelled subsequence still denoted $%
u_{\varepsilon _{n}}$, we may assume that $u_{\varepsilon _{n}}$ converges
a.e. in $\Omega .$ Then, owing to Dominated Convergence Theorem, we infer
that 
\begin{equation*}
u_{\varepsilon _{n}}\rightarrow u^{\star }\text{ \ in }L^{\delta
_{k}}(\Omega )\text{ \ for all }k\geq 1.
\end{equation*}%
Again, Dominated Convergence Theorem implies 
\begin{equation*}
u_{\varepsilon _{n}}|u_{\varepsilon _{n}}|^{a_{k}}f_{\varepsilon
_{n}}\rightarrow u^{\star }|u^{\star }|^{a_{k}}f(x,u^{\star },v^{\star })%
\text{ \ in }L^{1}(\Omega ).
\end{equation*}%
\newline
By Young's inequality we get 
\begin{equation*}
\Vert u_{\varepsilon _{n}}\Vert _{\delta _{k+1}}^{a_{k}+p}\leq {\int_{\Omega
}}u_{\varepsilon _{n}}|u_{\varepsilon _{n}}|^{a_{k}}f_{\varepsilon
_{n}}dx+|\Omega |+\Vert u_{\varepsilon _{n}}\Vert _{\delta _{k}}^{\delta
_{k}}.
\end{equation*}%
Passing to the limit one derives 
\begin{equation*}
\Vert u^{\star }\Vert _{\delta _{k+1}}^{a_{k}+p}\leq {\int_{\Omega }}%
u^{\star }|u^{\star }|^{a_{k}}f(x,u^{\star },v^{\star })dx+\left( |\Omega
|+\Vert u^{\star }\Vert _{\delta _{k}}^{a_{k}}\right) .
\end{equation*}%
By Remark \ref{R1}, we deduce 
\begin{equation*}
\Vert u^{\star }\Vert _{\delta _{k+1}}^{a_{k}+p}\leq {\int_{\Omega }}%
|u^{\star }|^{a_{k}}\left( |u^{\star }|^{pC}+|v^{\star }|^{qC}\right)
dx+\left( |\Omega |+\Vert u^{\star }\Vert _{\delta _{k}}^{a_{k}}\right) .
\end{equation*}%
Now, observe that 
\begin{equation*}
\dfrac{a_{k}}{\delta _{k}}+\dfrac{qC}{\gamma _{k}}=1-\dfrac{pC}{\delta _{k}}+%
\dfrac{qC}{\gamma _{k}}=1-\dfrac{1}{f_{k}}+\dfrac{1}{f_{k}}=1.
\end{equation*}%
Thus, using Young's inequality on the term $|v^{\star }|^{a_{k}}|v^{\star
}|^{qC},$ we get%
\begin{equation}
\Vert u^{\star }\Vert _{\delta _{k+1}}^{a_{k}+p}\leq A\left( 1+\Vert
u^{\star }\Vert _{\delta _{k}}^{\delta _{k}}+\Vert v^{\star }\Vert _{\gamma
_{k}}^{\gamma _{k}}\right) .  \label{lbk}
\end{equation}%
\medskip Similarly, by considering the component $v_{\varepsilon _{n}},$ we
obtain 
\begin{equation}
\Vert v^{\star }\Vert _{\gamma _{k+1}}^{b_{k}+q}\leq B\left( 1+\Vert
u^{\star }\Vert _{\delta _{k}}^{\delta _{k}}+\Vert v^{\star }\Vert _{\gamma
_{k}}^{\gamma _{k}}\right) .  \label{muk}
\end{equation}%
Observe that%
\begin{equation*}
\delta _{k+1}=a_{k+1}+pC=pC\left( DC^{k+1}+1\right) =pC\left( \dfrac{a_{k}}{p%
}+1\right) =C\left( a_{k}+p\right)
\end{equation*}%
and%
\begin{equation*}
\gamma _{k+1}=qCf_{k+1}=qCD\left( C^{k+1}+\dfrac{1}{D}\right) =C\left(
b_{k}+q\right)
\end{equation*}%
Thus 
\begin{equation*}
\Vert u^{\star }\Vert _{\delta _{k+1}}^{\delta _{k+1}/C}\leq A\left( 1+\Vert
u^{\star }\Vert _{\delta _{k}}^{\delta _{k}}+\Vert v^{\star }\Vert _{\gamma
_{k}}^{\gamma _{k}}\right) ,
\end{equation*}%
\begin{equation}
\Vert v^{\star }\Vert _{\gamma _{k+1}}^{\gamma _{k+1}/C}\leq B\left( 1+\Vert
u^{\star }\Vert _{\delta _{k}}^{\delta _{k}}+\Vert v^{\star }\Vert _{\gamma
_{k}}^{\gamma _{k}}\right) ,
\end{equation}%
that is 
\begin{equation}
\Vert u^{\star }\Vert _{\delta _{k+1}}^{\delta _{k+1}}\leq A^{C}\left(
1+\Vert u^{\star }\Vert _{\delta _{k}}^{\delta _{k}}+\Vert v^{\star }\Vert
_{\gamma _{k}}^{\gamma _{k}}\right) ^{C},  \label{ulb}
\end{equation}%
and 
\begin{equation}
\Vert v^{\star }\Vert _{\gamma _{k+1}}^{\gamma _{k+1}}\leq B^{C}\left(
1+\Vert u^{\star }\Vert _{\delta _{k}}^{\delta _{k}}+\Vert v^{\star }\Vert
_{\gamma _{k}}^{\gamma _{k}}\right) ^{C}.  \label{vlb}
\end{equation}%
Denote by $E_{k}=\Vert u^{\star }\Vert _{\delta _{k}}^{\delta _{k}}+\Vert
v^{\star }\Vert _{\gamma _{k}}^{\gamma _{k}}.$ Combining (\ref{ulb}) and (%
\ref{vlb}), it follows that 
\begin{equation*}
E_{k+1}\leq (A+B)^{C}E_{k}^{C}.
\end{equation*}%
We set $e_{k}=\ln {E_{k}},$ then we obtain the following iterative scheme 
\begin{equation}
e_{k+1}\leq C{\ln (A+B)}+C{e_{k}}.  \label{ek}
\end{equation}%
Proceeding by successive iterations, (\ref{ek}) can be formulated as follows 
\begin{equation*}
e_{k+1}\leq C^{k+1}\left( e_{0}+\dfrac{C}{C-1}\right) .
\end{equation*}%
Then we deduce that%
\begin{equation*}
\begin{array}{l}
\ln (\Vert u^{\star }\Vert _{\delta _{k+1}})\leq \dfrac{C^{k+1}}{\delta
_{k+1}}\left( e_{0}+\dfrac{C}{C-1}\right) \\ 
\leq \dfrac{C^{k+1}}{pCDC^{k+1}}\left( e_{0}+\dfrac{C}{C-1}\right) \leq 
\dfrac{1}{pCD}\left( e_{0}+\dfrac{C}{C-1}\right)%
\end{array}%
\end{equation*}%
and 
\begin{equation*}
\begin{array}{l}
\ln (\Vert v^{\star }\Vert _{\gamma _{k+1}})\leq \dfrac{C^{k+1}}{\gamma
_{k+1}}\left( e_{0}+\dfrac{C}{C-1}\right) \\ 
\leq \dfrac{C^{k+1}}{qCDC^{k+1}}\left( e_{0}+\dfrac{C}{C-1}\right) \leq 
\dfrac{1}{qCD}\left( e_{0}+\dfrac{C}{C-1}\right) .%
\end{array}%
\end{equation*}%
Else, the estimates hold 
\begin{equation*}
\begin{array}{l}
\Vert u^{\star }\Vert _{\infty }\leq {\limsup_{k\rightarrow +\infty }}%
exp\left( \dfrac{C^{k+1}\ln \Vert u^{\star }\Vert _{\delta _{k+1}}}{\delta
_{k+1}}\right) \\ 
\leq exp\left[ \dfrac{1}{pCD}\left( e_{0}+\dfrac{C}{C-1}\right) \right]%
\end{array}%
\end{equation*}%
and 
\begin{equation*}
\begin{array}{l}
\Vert v^{\star }\Vert _{\infty }\leq {\limsup_{k\rightarrow +\infty }}%
exp\left( \dfrac{C^{k+1}\ln \Vert v^{\star }\Vert _{\gamma _{k+1}}}{\gamma
_{k+1}}\right) \\ 
\leq exp\left[ \dfrac{1}{qCD}\left( e_{0}+\dfrac{C}{C-1}\right) \right] .%
\end{array}%
\end{equation*}%
Therefore 
\begin{equation}
\begin{array}{l}
{\max }(\Vert u^{\star }\Vert _{\infty },\Vert v^{\star }\Vert _{\infty })
\\ 
\leq {\min }\left( \dfrac{1}{pCD}\left( \ln {E_{\Theta }}+\dfrac{C}{C-1}%
\right) ,\,\dfrac{1}{qCD}\left( \ln {E_{\Theta }}+\dfrac{C}{C-1}\right)
\right) .%
\end{array}
\label{bound}
\end{equation}%
However, recall that $e_{0}=\ln {E_{0}},$ where $E_{0}=\Vert u^{\star }\Vert
_{\delta _{0}}^{\delta _{0}}+\Vert v^{\star }\Vert _{\gamma _{0}}^{\gamma
_{0}},$ and so, because the embeddings $W_{0}^{1,p}(\Omega )\hookrightarrow
L^{\delta _{0}}(\Omega )$ and $W_{0}^{1,q}(\Omega )\hookrightarrow L^{\gamma
_{0}}(\Omega )$ are continuous, more precisely, we also have $E_{0}\leq
\Vert u^{\star }\Vert _{1,p}^{\delta _{0}}+\Vert v^{\star }\Vert
_{1,q}^{\gamma _{0}}.$ Since the proof of the first assert in Proposition %
\ref{Tta} remains valid by taking $\tau =1,$ then there exists a constant $%
C_{\Theta },$ depending only on $\Theta ,$ such that $E_{0}\leq C_{\Theta }$.

Consequently, the right-hand side in (\ref{bound}) is independent of $%
(u^{\star },v^{\star }).$ The proof is complete.
\end{proof}

\end{document}